\documentclass[a4paper,draft]{amsart}
\usepackage[utf8]{inputenc}
\usepackage[ngerman, english]{babel}
\usepackage{amsmath,amsthm,amssymb,amsfonts}
\usepackage{mathrsfs}
\usepackage{ifthen}
\usepackage{enumerate}
\usepackage{hyphsubst}
\usepackage{xcolor}
\usepackage{graphicx}
\usepackage{psfrag}
\usepackage{pst-all}
\usepackage{xcolor}

\newtheorem{theorem}{Theorem}[section]
\newtheorem{lemma}[theorem]{Lemma}

\newtheorem{proposition}[theorem]{Proposition}
\newtheorem{definition}[theorem]{Definition}

\newtheorem{remark}[theorem]{Remark}

\newtheorem{innercustomthm}{Theorem}
\newenvironment{customthm}[1]
  {\renewcommand\theinnercustomthm{#1}\innercustomthm}
  {\endinnercustomthm}

\newcommand{\nv}{\boldnu}
\newcommand{\domain}{D}
\newcommand{\angvel}{\Omega}
\newcommand{\vel}{\boldv}
\newcommand{\dens}{\rho}
\newcommand{\densl}{\ul{\rho}}
\newcommand{\densu}{\ol{\rho}}
\newcommand{\pres}{p}
\newcommand{\grav}{\phi}
\newcommand{\dislag}{\boldu}
\newcommand{\dislagscal}{u}
\newcommand{\graveu}{\psi}
\newcommand{\bflow}{\boldb}
\newcommand{\qcoeff}{\boldq}
\newcommand{\sounds}{c_s}
\newcommand{\soundsl}{\ul{c_s}}
\newcommand{\soundsu}{\ol{c_s}}
\newcommand{\rchange}{\Gamma}
\newcommand{\damp}{\gamma}
\newcommand{\dampl}{\ul{\gamma}}
\newcommand{\dampu}{\ol{\gamma}}
\newcommand{\ses}{a}
\newcommand{\sescow}{\ses_\mathrm{Cow}}
\newcommand{\sesschur}{\ses_\mathrm{Schur}}
\newcommand{\sesadd}{\ses_{\Delta}}
\newcommand{\op}{A}
\newcommand{\opAcow}{\op_\mathrm{Cow}}
\newcommand{\opdecomp}{B}
\newcommand{\opR}{R}
\newcommand{\opEmbed}{E}
\newcommand{\opQ}{Q}
\newcommand{\hs}{\boldX}
\newcommand{\hsV}{\boldV}
\newcommand{\hsW}{\boldW}
\newcommand{\hsZ}{\boldZ}
\newcommand{\hsgeneric}{Y}
\newcommand{\ugeneric}{y}
\newcommand{\hsgrav}{\tilde H^1_*}

\newcommand{\blo}{L}
\newcommand{\compv}{\boldv}
\newcommand{\compvp}{v_0}
\newcommand{\compw}{\boldw}

\newcommand{\compz}{\boldz}
\newcommand{\PV}{P_\hsV}
\newcommand{\PW}{P_\hsW}
\newcommand{\PZ}{P_\hsZ}
\newcommand{\Pran}{P_{\ran\opdecomp^\bot}}
\newcommand{\Creg}{C_\mathrm{reg}}
\newcommand{\matgr}{\ul{\ul{m}}}
\newcommand{\matkl}{M}
\newcommand{\complsc}{\sigma}
\newcommand{\DLag}{\delta_{\mathrm{L}}}
\newcommand{\DEu}{\delta_{\mathrm{E}}}

\newcommand{\dd}{\mathrm{d}}
\newcommand{\ol}[1]{\overline{#1}}
\newcommand{\ul}[1]{\underline{#1}}
\newcommand{\spl}{\langle}
\newcommand{\spr}{\rangle}
\newcommand{\bpm}{\begin{pmatrix}}
\newcommand{\epm}{\end{pmatrix}}

\renewcommand{\div}{\operatorname{div}}
\DeclareMathOperator{\grad}{grad}
\DeclareMathOperator{\curl}{curl}
\DeclareMathOperator{\hess}{Hess}

\renewcommand{\dim}{\operatorname{dim}}
\DeclareMathOperator{\ran}{ran}

\DeclareMathOperator{\sign}{sgn}

\DeclareMathOperator{\numran}{num ran}

\newcommand{\boldnu}{\boldsymbol{\nu}}
\newcommand{\boldxi}{\boldsymbol{\xi}}

\newcommand{\setC}{\mathbb{C}}

\newcommand{\setR}{\mathbb{R}}

\newcommand{\boldb}{\mathbf{b}}

\newcommand{\boldf}{\mathbf{f}}
\newcommand{\boldg}{\mathbf{g}}

\newcommand{\boldq}{\mathbf{q}}

\newcommand{\boldu}{\mathbf{u}}
\newcommand{\boldv}{\mathbf{v}}
\newcommand{\boldw}{\mathbf{w}}
\newcommand{\boldx}{\mathbf{x}}
\newcommand{\boldy}{\mathbf{y}}
\newcommand{\boldz}{\mathbf{z}}
\newcommand{\boldA}{\mathbf{A}}

\newcommand{\boldG}{\mathbf{G}}
\newcommand{\boldH}{\mathbf{H}}

\newcommand{\boldL}{\mathbf{L}}

\newcommand{\boldV}{\mathbf{V}}
\newcommand{\boldW}{\mathbf{W}}
\newcommand{\boldX}{\mathbf{X}}

\newcommand{\boldZ}{\mathbf{Z}}

\newcommand{\calT}{\mathcal{T}}

\usepackage[pdfauthor={Martin Halla and Thorsten Hohage},pdftitle={On the well-posedness of the damped time-harmonic Galbrun equation and the equations of stellar oscillations},
unicode,colorlinks=true,linkcolor=black,citecolor=black,urlcolor=black,pagebackref=false,breaklinks]{hyperref}

\title[Well-posedness of time-harmonic Galbrun equation]{On the well-posedness of the damped time-harmonic Galbrun equation 
and the equations of stellar oscillations}
\author{Martin Halla and Thorsten Hohage}
\email{halla@mps.mpg.de}
\date{\today}

\begin{document}
\begin{abstract}
We study the time-harmonic Galbrun equation describing the propagation of sound in the presence of a steady background flow. 
With additional rotational and gravitational terms these equations are also fundamental in helio- and asteroseismology as a model 
for stellar oscillations. 
For a simple damping model we prove well-posedness of these equations, i.e.\ uniqueness, existence, and stability of solutions
under mild conditions on the parameters (essentially subsonic flows). The main tool of our analysis is a generalized Helmholtz decomposition. 
\end{abstract}
\maketitle

\section{Introduction}\label{sec:introduction}

In this paper we study the equations for time-harmonic acoustic waves in the presence of a steady background flow $\bflow$ 
and damping effects. Then the Lagrangian perturbations $\dislag$ of displacement satisfy Galbrun's equation
\begin{subequations}\label{eqs:Galbrun}
\begin{align}\label{eq:Galbrun}
\begin{aligned}
\dens(-i\omega+\partial_{\bflow})^2\dislag 
&- \grad\left(\dens\sounds^2\div \dislag\right)
+ (\div \dislag) \grad \pres\\
&-\grad(\grad\pres\cdot \dislag) + \hess(\pres)\dislag 
+ \damp \dens (-i \omega
) \dislag
= \boldf \quad \mbox{in }\domain
\end{aligned}
\end{align}
where $\dens, \pres, \sounds,\bflow$ and $\boldf$ denote density, pressure, sound speed, background velocity and sources, 
$\partial_\bflow := \sum_{l=1}^3 \bflow_l\partial_{x_l}$  
denotes the directional derivative in direction $\bflow$, 
$\hess(\pres)$ the Hessian of $\pres$, $\domain\subset\mathbb{R}^3$ a bounded Lipschitz domain, 
and damping is modeled by the term $- i \omega \damp \dens \dislag$ 
with damping coefficient $\damp$.  Eq.~\eqref{eq:Galbrun} has its roots in the time-dependent non-linear Euler equations. 
It was observed by Galbrun \cite{Galbrun:31} that the Lagrangian linearization of these equations can be reformulated in an advantageous way, which reduces the number of unknowns. 
Galbrun's equation \eqref{eq:Galbrun} and its time-dependent analog are used in aeroacoustics to model (and eventually reduce) noise caused by moving objects such as air-conditioning devices, cars or aircraft engines (see, e.g.\ \cite{GLT:04}). 
There exist several related equations such as  Goldstein's or 
M\"ohring's equations, for which we refer to the discussions in \cite{BensalahJolyMercier:18,HaeggBerggren:19} and references therein. 	

In this paper we will only treat boundary conditions of the form  
\begin{align}\label{eq:bc_Galbrun}
\dislag\cdot \nv =g\qquad \mbox{on }\partial\domain\,.
\end{align}
\end{subequations}
The case of natural boundary conditions $\div\dislag=g$ on $\partial\domain$ poses no substantial differences in the analysis.
Here and in the following 
$\nv$ denotes the exterior unit normal vector on $\partial \domain$. 
To treat unbounded domains, the boundary condition \eqref{eq:bc_Galbrun} 
has to be complemented by a radiation condition at infinity, 
or replaced by  
a transparent boundary condition on artificial boundaries of a computational domain. The construction of such transparent boundary conditions is a topic of independent interest, which we will not touch here (see \cite{BecacheBonnetBDLegendre:06}).  

This work has mainly been motivated by the equations of stellar oscillations, a generalization of Galbrun's equation, which were first derived in \cite{LyndenBOstriker:67}. They are 
formulated in a frame rotating at constant angular velocity 
$\angvel\in\mathbb{R}^3$ with the star and therefore involve additional rotational terms. Moreover, the (scaled) Eulerian perturbation 
$\graveu$ of the gravitational background potential $\phi$ is needed as an additional scalar unknown to capture important types of waves in stars such as $g$- and $f$-modes. Stellar oscillations can be described by the following system 
of differential equations for the unknowns $\dislag$ and $\graveu$:  
\begin{subequations}\label{eqs:stellar_osc}
\begin{align}\label{eq:stellar_osc_dislag}
\begin{aligned}
\dens(-i\omega+\partial_{\bflow}+\angvel\times)^2\dislag 
- \grad(\dens\sounds^2\div\dislag)
+ (\div \dislag) \grad\pres
-\grad(\grad\pres\cdot \dislag) 
\qquad&\\
 + \left(\hess(\pres)-\dens\hess(\phi)\right)\dislag 
+\damp \dens(- i \omega
) \dislag
+\dens \grad\graveu
= \boldf\quad& \mbox{in }\domain
\end{aligned}\\
\label{eq:stellar_osc_grav}
-\frac{1}{4\pi G}\Delta \graveu + \div(\dens\dislag) =  0\quad\mbox{in }\mathbb{R}^3&
\end{align}
Here $G$ denotes the gravitational constant, and $\dislag$ is set to zero on $\mathbb{R}^3\setminus \domain$. 
For a thorough discussion of the physical background of these equations and of their properties 
such as eigen-decomposition under radial symmetry we refer to the monographs \cite{HS:11,UOAS:89}. 
A sketch of the derivation of these equations is included in Sections \ref{sec:nonlinear}--\ref{sec:timeharmonic}. 
We have modeled the important effect of damping by the simple term $\damp \dens(- i \omega) \dislag$ here. 
The most important mechanisms contributing to wave attenuation are believed to be radiative damping and interaction 
with turbulent convection (see \cite{UOAS:89}). 
More refined models of these processes will have to be developed and analyzed in the future. 
Another simple model of wave attenuation, which is convenient for a modal analysis, is to replace $\omega$ by $\omega+i\gamma$ 
(see \cite{GB:02}). 
Note that in comparison to \eqref{eq:stellar_osc_dislag} this leads to additional terms after expanding the square. 

As for Galbrun's equation we only consider boundary conditions 
\begin{align}\label{eq:stellar_osc_bd_dislag}
\nv\cdot\dislag = g& \quad \mbox{on }\partial \domain
\end{align}
for $\dislag$ as well as a decay condition for $\graveu$:
\begin{align}\label{eq:stellar_osc_decay}
\lim_{|\boldx|\to \infty} \graveu(\boldx)=0.
\end{align}
\end{subequations}

The equations \eqref{eqs:stellar_osc} form the basis of helio- and asteroseismology (as opposed to seismology of the Earth, which 
is based on the linearized elasticity equations, at least 
in the solid mantle and crust). In asteroseismology and global helioseismology 
one aims to infer values of coefficients in a radially symmetric model from 
observations of eigen-frequencies. Whereas substantial simplifications apply in these applications due to the radial symmetry assumption, 
in local helioseismology, which aims at two- or three-dimensional imaging of physical quantities (in particular flows) in the solar interior, 
the full equations~\eqref{eqs:nl_pde} are often used for the definition of forward problems (see \cite{GB:02}). 
Solar oscillations (Sunquakes) excited by turbulent convection in the outer convection zone have been observed continuously and at high resolution over more than 20 years by satellite and ground-based Doppler shift measurements (see \cite{GBS:10} for more information). 
A simpler model based on a Helmholtz-type equation for 
the scalar quantity $\sounds\div\dislag$ as new unknown has been suggested in \cite{gizon_etal:17}. 
The scalar model in \cite{gizon_etal:17} cannot account for 
f- and g-modes, its validity is restricted to coefficients $\dens$, $\sounds$ and $\damp$ varying slowly compared to the wave-length. 
Moreover, it is not obvious how to relate $\sounds\div\dislag$ on the solar surface to the observed Doppler shift data. 
Therefore, it would be highly desirable to work with the 
full vectorial equation \eqref{eqs:stellar_osc} instead, and the 
aim of this paper is to contribute theoretical foundations for such an approach.

\medskip
It seems that surprisingly little is known about the fundamental question of well-posedness of the boundary value-problems 
\eqref{eqs:Galbrun} and \eqref{eqs:stellar_osc}, i.e.\ uniqueness, existence, and stability of solutions to these or related differential equations. 
We are only aware of two references which report results on the well-posedness of Galbrun's equation with non-uniform flow.
The first one \cite{BonnetBDMercierMillotPernetPeynaud:12} considers the time-harmonic case. Its approach is to regularize the equation for the displacement and to derive an additional (transport) equation for the vorticity. In two space dimensions, for certain kinds of flows $\bflow$ with small supremum norm of all first derivatives, 
the authors obtain Fredholmness of the system.
The second one \cite{HaeggBerggren:19} considers the time-dependent case. Its approach is to obtain well-posedness for the Eulerian linearization and subsequently to construct a solution for Galbrun's equation from the former. In contrast to these references, our approach is a purely ``elliptic'' one without any ``transport equation techniques''.

Mathematical analysis of the equations of stellar oscillations seems to have focused mainly on the time domain equations, studying long-time (secular) stability in the absence of damping and on completeness of normal modes including potential contributions from continuous parts of the  spectrum (see \cite{beyer:02,DS:79,hunter:77}). 

It has been known (see \cite[Appendix A]{hunter:77} and 
\cite[eq.~(3.16)]{DS:79}) that the sequilinear forms associated to the operators in \eqref{eqs:Galbrun} and \eqref{eqs:stellar_osc} without the first terms involving $i\omega$ are bounded from below by a (possibly negative) muliple of the squared $\boldL^2$-norm. 
If the associated energy space was compactly embedded in $\boldL^2$ and if $\bflow=0$, this would yield a G{\aa}rding-type inequality implying the desired Fredholmness of the differential operator.  
To deal with the lack of such a compact embedding, we derive 
a generalized Helmholtz decomposition which allows to establish 
weak $T$-coercivity for an operator $T$ flipping the sign of one of the
components of this Helmholtz decomposition. 
By this means we obtain the well-posedness of the equations under quite weak assumptions: We
only require a $H(\div)$-type regularity of the background flow $\bflow$ and small enough $\boldL^\infty$-norm, see Theorem~\ref{thm:wTc-full}. For sufficiently smooth 
$\dens$, $\sounds$, $\domain$, 
 and homogeneous pressure $\pres$ and gravity $\grav$, our 
smallness assumption on $\bflow$ becomes 
\begin{align*}
\|\sounds^{-1}\bflow\|_{\boldL^\infty}<1,
\end{align*}
i.e.\ that the flow is everywhere subsonic. We demonstrate in one space dimension that the properties of the systems change considerably if this assumption is violated. 
For smooth, but not necessarily constant pressure and gravity, the constant on the right hand side may be 
smaller than $1$, but tends to $1$ as $\omega\to\pm\infty$. 
We stress that our results require very little smoothness of the flow
and no assumption on the geometry of the flow (e.g.\ that ``$\bflow$ is filling''). These very weak assumptions on $\bflow$ may, e.g., be advantageous in the context of iterative methods for inverse problem with the flow as unknown.

The remainder of this article is structured as follows. In Section~\ref{sec:derivation} we derive the formerly presented sesquilinear form, define the corresponding Hilbert space and formulate our basic assumptions on the parameters. At the beginning of Section~\ref{sec:analysis} we set our notation. In Subsection~\ref{subsec:helmholtz} we prove the existence of a generalized Helmholtz decomposition, which is suited for our sesquilinear form. See Theorem~\ref{thm:decomposition}. In Subsection~\ref{subsec:cowling} we introduce the so-called Cowling approximation, which reduces the unknowns $(\dislag,\graveu)$ to $\dislag$. 
We explain that the Cowling approximation is a suitable simplification to study the original equation. In Subsection~\ref{subsec:bflow} we consider the special case of homogeneous pressure and gravity. 
In Subsection~\ref{subsec:pres_grav} we consider the special case of no flow. In Subsection~\ref{subsec:general} we merge the two kinds of analyses for the case of general parameters. 
Finally, in Subsection~\ref{subsec:fullequations} we consider the original equation and report our main results in Theorem~\ref{thm:wTc-full} 
before we end the paper with some conclusions in Section~\ref{sec:conclusion}. 
An appendix discusses variations of our main results for rough data under more stringent smallness assumptions.

\section{Derivation and weak formulation}\label{sec:derivation}
To help the reader finding a way through several versions of the basic equations discussed in the literature, we briefly 
sketch the derivation of Eq.~\eqref{eqs:stellar_osc} and then derive a weak formulation. 

\subsection{The non-linear equations}\label{sec:nonlinear}
As usual the time coordinate is denoted by $t\in\setR$ and the spatial coordinate by $\boldx\in\setR^3$.
Following Lynden-Bell \& Ostriker \cite{LyndenBOstriker:67} we consider a fluid described by
the velocity $\vel(t,\boldx)\in\setR^3$, the density $\dens(t,\boldx)\in\setR$, the pressure $\pres(t,\boldx)\in\setR$, the gravitational potential $\grav(t,\boldx)\in\setR$ in a coordinate system rotating 
at fixed angular velocity $\angvel\in\setR^3$.
The fluid is excited by an external force $\tilde\boldf(t,x)\in\setR^3$.
Then the conservation of momentum in this rotating frame is described by the generalized 
Euler equations 
\begin{subequations}\label{eqs:nl_pde}
\begin{align}\label{eq:motion}
(\partial_t+\partial_\vel)\vel +2\angvel\times\vel+\angvel\times(\angvel\times\boldx)
 =  \grad\grav -\frac{1}{\dens} \grad\pres + \frac{1}{\dens}\tilde\boldf \quad\text{in }\domain,
\end{align}
the conservation of mass by the continuity equation
\begin{align}\label{eq:cont}
\partial_t \dens +\div(\dens\vel)=0 \quad\text{in }\domain,
\end{align}
and the gravitational potential satisfies the equation 
\begin{align}\label{eq:gravity}
-\frac{1}{4\pi G}\Delta \grav = \dens \qquad \mbox{in }\setR^3
\end{align}
whereby $\dens$ is set to zero in $\setR^3\setminus\domain$ and $G$ is the gravitational constant. Moreover, we impose the boundary condition
\begin{align}\label{eq:bc_stellar_osc}
\nv\cdot\vel=0 \quad\text{at }\partial\domain,
\end{align}
and the decay condition
\begin{align}\label{eq:gravity-rc}
\lim_{|\boldx|\to \infty} \grav(\boldx) = 0.
\end{align}
\end{subequations}

\subsection{The linear time-harmonic equation}\label{sec:timeharmonic}
We consider a (sufficiently smooth) stationary equilibrium solution. That is $(\vel_0,\dens_0,\pres_0,\grav_0)$ solve Eqs.~\eqref{eqs:nl_pde} and 
$\partial_t\vel_0=0$, $\partial_t\dens_0=\partial_t \pres_0=\partial_t \grav_0=0$.
For non-stationary solutions $(\vel,\dens,\pres,\grav)$ to \eqref{eqs:nl_pde}, which are 
``close'' to $(\vel_0,\dens_0,\pres_0,\grav_0)$, we define the Eulerian perturbations by 
\begin{align}
\DEu \varphi(t,\boldx):= \varphi(t,\boldx)-\varphi_0(t,\boldx),\qquad \varphi \in \{\vel,\dens,\pres,\grav\}.
\end{align}
The path of a fluid particle, which is at position $\boldy\in\domain$ at time 
$t=0$ is described by the solution $X(\cdot,\boldy)$ to the initial value problem
\[
\partial_t X(t,\boldy) = \vel(t,X(t,\boldy)),\qquad X(0,\boldy)=\boldy.
\]
The Lagrangian perturbations are defined by
\begin{align}\label{eq:Lagr_pert}
\DLag \varphi(t,X_0(t,\boldy)):= \varphi(t,X(t,\boldy))-\varphi_0(t,X_0(t,\boldy)),\qquad 
\end{align}
where $X_0$ is defined by the same initial value problem as $X$ with $\vel$ replaced by $\vel_0$. 
Our aim is to derive linear approximate equations for such perturbations. 
Let us abbreviate the Lagrangian perturbation of displacement $\phi_d(t,\boldx):=\boldx$ by 
$\dislag:=\DLag\phi_d$, i.e. 
\[
\dislag(t,X_0(t,\boldy)):= X(t,\boldy)- X_0(t,\boldy).
\]
(Often the symbol $\mathbf{\xi}$ is used instead of $\dislag$.) 
Then the definition of $\DLag\varphi$ may be rewritten as 
\[
\DLag \varphi(t,\boldx):= \varphi(t,\boldx + \dislag(t,\boldx))-\varphi_0(t,\boldx).
\]
It follows that the Eulerian and Lagrangian perturbations are related 
by
\begin{align}\label{eq:EulerLag}
\DLag\varphi \approx \DEu\varphi + \dislag\cdot \grad\varphi.
\end{align}
Here and in the remainder of this subsection the relation $\approx$ denotes equality up to first order terms in $\dislag$.
Subtracting \eqref{eq:motion} with $\vel=\vel_0$ from \eqref{eq:motion} with 
$\boldx$ replaced by $\boldx + \dislag(t,\boldx)$ and multiplying by $\dens_0$ yields 
\begin{align}\label{eq:motionLag}
\begin{aligned}
\dens_0 \Big[\DLag(\partial_t+\partial_\vel)\vel 
+2\angvel\times\DLag\vel+\angvel\times(\angvel\times\dislag)\Big]
&= \dens_0\DLag\left[\grad\grav -\frac{1}{\dens} \grad\pres 
+ \frac{1}{\dens}\tilde{\boldf}\right]\,.
\end{aligned}
\end{align}
A straightforward computation (see \cite[eq.~(90)]{HaeggBerggren:19}) shows that 
\[
\DLag \vel = (\partial_t+ \partial_{\vel_0})\dislag\,.
\]
Moreover, as shown in \cite{LyndenBOstriker:67}, we have
\[
\DLag\left(\partial_t+\partial_{\vel}\right) \approx \left(\partial_t+\partial_{\vel_0}\right)\DLag
\]
up to first order terms in $\dislag$.
With the help of these two identities, the left hand side of \eqref{eq:motionLag} simplifies to 
\[
\dens_0\left(\partial_t+\partial_{\vel_0} + \angvel \times \right)^2\dislag.
\]
The right hand side can be simplified as in \cite{LyndenBOstriker:67}. In particular, it 
follows from the continuity equation \eqref{eq:cont} that 
\begin{align}\label{eq:lin_cont}
&\DEu \dens + \div(\dens_0\dislag) \approx 0 
&&\DLag \dens + \dens_0\div\dislag \approx 0 
\end{align}
to first order 
(see \cite[eq.~(19)]{LyndenBOstriker:67}). Moreover, density and pressure can be related by an 
additional energy equation, which for adiabatic changes leads to a proportionality 	
\begin{align}\label{eq:rchange}
\frac{\DLag \pres}{\pres_0}= \rchange(\boldx)\frac{\DLag \dens}{\dens_0}\,.
\end{align}
Using \eqref{eq:lin_cont} and \eqref{eq:rchange}, 
 $\DLag\pres$ and $\DLag\dens$ can be eliminated from the expansion of 
$\DLag[\grad\pres/\dens]$, and one arrives at 
\begin{align*}
\dens_0\DLag\left[\frac{1}{\dens}\right.&\left.\grad\pres\right] 
\approx\\
&-\grad\left( \rchange\pres_0 \div \dislag\right)
+\grad \left(\pres_0 \div \dislag\right) -\grad(\grad\pres_0\cdot\dislag)+\hess(\pres_0)\dislag.
\end{align*}
Moreover, it follows from \eqref{eq:EulerLag}  that 
\begin{align*}
\dens_0\DLag\left[\grad\grav\right] 
\approx \dens_0 \hess(\grav_0) \dislag+\dens_0\grad\DEu\grav\,.
\end{align*}
Introducing the speed of sound by $\sounds:=\sqrt{\rchange\pres_0/\dens_0}$ 
and setting 
\[
\graveu := \DEu\grav
\]
and $\boldf :=\dens_0\DLag(\tilde{\boldf}/\dens)$ (note that $\DLag(\tilde{\boldf}/\dens)$ may contain dependencies of $\dislag$ which are not considered here), and taking a Fourier transform in time,  
eq.~\eqref{eq:motionLag} becomes 
\begin{subequations}\label{eqs:solar_osc_freq}
\begin{align}
\begin{aligned}
&\dens_0\left(-i\omega+\partial_{\vel_0} + \angvel \times \right)^2\widehat{\dislag}
=\grad\left(\sounds^2\dens_0 \div \widehat{\dislag}\right)
-\grad\left(\pres_0 \div \widehat{\dislag}\right) \\
&\qquad\qquad\qquad +\grad(\grad \pres_0\cdot\widehat{\dislag})
-\hess(\pres_0)\widehat{\dislag} 
+\dens_0\hess(\grav_0)\widehat{\dislag} 
+\dens_0\nabla \widehat{\graveu}+\widehat{\boldf}
\end{aligned}
\end{align}
where $\widehat{\dislag}(\boldx)$ denotes the Fourier transform of $\dislag(\cdot,\boldx)$ evaluated 
at $\omega$, and analogously for $\widehat{\grav}$ and $\widehat{\boldf}$.
At this point we add the damping term $-i\omega\gamma \rho \widehat{\dislag}$ to the left-hand-side of \eqref{eqs:solar_osc_freq} which models attenuation by gravitational radiation and viscosity \cite{FS:75b, DS:79}. We note that this kind of damping 
differs from the damping model in \cite[eq.~(1)]{GizonEtal:17} where
damping is modelled by replacing $\omega$ by $\omega+i\gamma$.\\
Applying $\DEu$ to the continuity equation and using \eqref{eq:lin_cont} leads to 
$-( 4\pi G)^{-1}\Delta \DEu\grav=\DEu\dens 
\approx -\div(\dens_0\dislag)$. In the frequency 
domain this yields  
\begin{align}
\label{eq:gravlag}
-\frac{1}{4\pi G}\Delta \widehat{\graveu} +\div(\dens_0\widehat{\dislag}) &= 0\,.
\end{align}
Moreover, it follows from \eqref{eq:EulerLag} and \eqref{eq:bc_stellar_osc} that 
\[
\nv\cdot\dislag = \nv \cdot \DLag \vel 
= \DEu\left[\nv\cdot \vel\right] + \nv\cdot\partial_{\vel_0}\vel_0
= \nv\cdot\partial_{\vel_0}\vel_0
\]
which yields the boundary condition 
\begin{align}
\nv\cdot \widehat{\dislag} = \nv\cdot\partial_{\vel_0}\vel_0
\qquad \mbox{on }\partial \domain.
\end{align}
Finally, \eqref{eq:gravity-rc} leads to 
\begin{align}
\lim_{|\boldx|\to \infty} \widehat{\grav}(\boldx) = 0.
\end{align}
\end{subequations}
Since from this point on we do not consider the time-dependent equations 
anymore, we drop the hats of $\widehat\dislag$, $\widehat\graveu$, $\widehat\boldf$ and there will occur no confusion in the overloaded notation. 
We also rename the velocity $\vel_0$ as background flow $\bflow$
\begin{align*}
\bflow:=\vel_0.
\end{align*}
and we drop the index $0$ of the remaining quantities, i.e.\
\begin{align*}
\dens:=\dens_0, \qquad \pres:=\pres_0, \qquad \grav:=\grav_0.
\end{align*}
In this notation the equations \eqref{eqs:solar_osc_freq} coincide with 
\eqref{eqs:stellar_osc} with $g= \nv\cdot\partial_{\vel_0}\vel_0$.

\subsection{The variational formulation}
We introduce the scalar products
\begin{align*}
\spl \dislagscal,\dislagscal' \spr:=\int_\domain \dislagscal\ol{\dislagscal'} \,\dd\boldx, \qquad
\spl \dislag,\dislag' \spr:=\int_\domain \dislag\cdot\ol{\dislag'} \,\dd\boldx,
\end{align*}
for scalar functions $\dislagscal,\dislagscal'$ and vectorial functions $\dislag,\dislag'$. Here $\ol{\cdot}$ denotes complex conjugation. Since we use the same symbol for the scalar products of scalar and vectorial functions the notation is overloaded, but its meaning will always be clear from the arguments.
Henceforth we consider $\domain$ as the default domain for all function spaces and suppress the dependency in the notation, if the domain equals $\domain$. Thus we write e.g.\ $L^2=L^2(\domain)$ and so on.
If we do not explicitly indicate a particular field, all spaces are over $\setC$, e.g.\ $L^2=L^2(\domain)=L^2(\domain;\setC)$. Further, we set $\boldL^2:=(L^2)^3$.

Next, to prepare the derivation of a variational formulation of \eqref{eqs:stellar_osc} we consider the following auxiliary steps.
Basic algebraic computations yield
$
(\angvel \times \dislag) \cdot \dislag' =
- \dislag \cdot (\angvel \times \dislag').
$
By means of integration by parts we compute
\begin{align*}
\spl \dens \partial_\bflow \dislag,\dislag' \spr &= -\spl \dens \dislag, \partial_{\ol{\bflow}} \dislag' \spr
-\spl \div(\dens\bflow) \dislag, \dislag' \spr
+ \int_{\partial\domain} (\nv\cdot\bflow) (\dislag \cdot \ol{\dislag'}) \,\dd\boldx\\
&= -\spl \dens \dislag, \partial_{\ol{\bflow}} \dislag' \spr.
\end{align*}
due to \eqref{eq:bc_stellar_osc} and \eqref{eq:cont}.
Since $\omega$, $\angvel$, $\bflow(\boldx)$ are real (vector) valued, the operators
$\omega$, $i\angvel\times$, $i\partial_\bflow$ are selfadjoint with respect to $\spl\dens\cdot,\cdot\spr$, i.e.
\begin{align*}
\spl \dens\omega\dislag,\dislag' \spr = \spl \dens\dislag,\omega\dislag' \spr,\quad 
\spl i\angvel \times \dislag,\dislag' \spr = \spl \dislag,i\angvel \times \dislag' \spr,\quad 
\spl \dens i\partial_\bflow \dislag,\dislag' \spr = \spl \dens \dislag, i\partial_\bflow \dislag' \spr.
\end{align*}
Thus
\begin{align*}
\spl \dens\big(\omega +i\partial_{\bflow}+i\angvel\times \big)^2 \dislag, \dislag'\spr = 
\spl \dens\big(\omega +i\partial_{\bflow}+i\angvel\times \big) \dislag,\big(\omega +i\partial_{\bflow}+i\angvel\times \big) \dislag'\spr,
\end{align*}
i.e.\ $\big(\omega +i\partial_{\bflow}+i\angvel\times \big)^2$ is selfadjoint with respect to $\spl\dens\cdot,\cdot\spr$.
Now, consider a solution $(\dislag,\graveu)$ to \eqref{eqs:stellar_osc}. If we test \eqref{eq:stellar_osc_dislag} and \eqref{eq:stellar_osc_grav}
with any $(\dislag',\graveu')$ such that $\dislag'$ satisfies $\nv\cdot\dislag'=0$ at $\partial\domain$, we obtain
\begin{align}\label{eq:variationaleq}
\ses\big((\dislag,\graveu),(\dislag',\graveu')\big) = \spl \boldf,\dislag' \spr
\end{align}
with the sesquilinear form
\begin{align}\label{eq:sesq-a}
\begin{split}
\ses\big((\dislag,\graveu),(\dislag',\graveu')\big)
&:=\spl \sounds^2\dens \div \dislag, \div \dislag' \spr
+\spl \div \dislag, \grad\pres \cdot \dislag' \spr
+\spl \grad\pres \cdot \dislag, \div \dislag' \spr\\
&+\spl (\hess(\pres)-\dens\hess(\grav)) \dislag,\dislag' \spr\\
&-\spl \dens (\omega+i\partial_\bflow+i\angvel\times) \dislag, (\omega+i\partial_\bflow+i\angvel\times) \dislag' \spr\\
&-i\omega \spl \damp\dens \dislag, \dislag' \spr\\
&- \spl \grad\graveu, \dens\dislag' \spr
- \spl \dens\dislag, \grad\graveu' \spr
+\frac{1}{4\pi G}\spl \nabla \graveu, \nabla \graveu' \spr.
\end{split}
\end{align}
Consequently, we define the Hilbert space
\begin{subequations}\label{eq:hilspX}
\begin{align}
\hs_{\bflow}&:=\{\dislag \in \boldL^2\colon \quad\div \dislag \in L^2, \quad \partial_\bflow \dislag \in \boldL^2,
\quad \nv\cdot\dislag=0\text{ at }\partial \domain\},
\end{align}
where the subscript $\bflow$ will usually be omitted, with inner product
\begin{align}
\spl \dislag,\dislag' \spr_\hs &:= \spl \div \dislag, \div \dislag' \spr
+\spl \partial_\bflow\dislag, \partial_\bflow\dislag' \spr
+\spl \dislag, \dislag' \spr
\end{align}
\end{subequations}
(see Lemma \ref{lem:HS}).
The appropriate space for the gravitational potential $\graveu$ is a bit more technical
since the $L^2(\setR^3)$-norm of $\graveu$ cannot be bounded by means of the sesquilinear form $\ses(\cdot,\cdot)$.
By the standard Helmholtz decomposition the set $\boldG:=\{\boldg\in \boldL^2(\setR^3)\colon \curl\boldg=0\}$ is a closed subspace of $\boldL^2(\setR^3)$ and hence a Hilbert space with the $L^2$-inner product. For each $\boldg\in\boldG$ exists a unique gradient potential $\graveu \in H^1_\mathrm{loc}(\setR^3)/\setC$ such that $\boldg=\nabla\graveu$. We define
\begin{align}\label{eq:hsgrav}
\hsgrav:=\{\graveu\colon \nabla\graveu\in\boldG\}, \qquad \spl\graveu,\graveu'\spr_{\hsgrav}:=\spl\nabla\graveu,\nabla\graveu'\spr_{\boldL^2(\setR^3)}
\end{align}
as appropriate  Hilbert space for the gravitational potential $\graveu$.
Hence for reasonable parameters ($\domain, \omega, \angvel, \sounds, \bflow, \dens, \pres, \grav, \damp$) an appropriate space for \eqref{eq:sesq-a} will be $\hs\times \hsgrav$. We will detail our assumptions on the parameters in the following.

\subsection{Basic Assumptions}\label{subsec:assumptions}
Let $\domain$ be a bounded Lipschitz domain.
Let $\omega\in\setR$ and $\angvel\in\setR^3$.
Let $\sounds, \dens, \damp\colon\domain\to\setR$ be measurable and such that there exist positive constants $\soundsl, \soundsu,\densl, \densu, \dampl, \dampu$ with 
\begin{align}\label{eq:box_constraints}
\densl \leq \dens \leq \densu, \qquad \soundsl \leq \sounds \leq \soundsu, \qquad \dampl \leq \damp \leq \dampu
\end{align}
almost everywhere in $\domain$. Let $\pres, \grav \in W^{2,\infty}(D,\setR)$.
Let $\bflow\in \boldL^{\infty}(\domain,\setR^3)$ be such that
$\div(\dens\bflow)\in L^2(\domain)$.
We require the latter assumption to well define the distributional derivate
$\dens\partial_\bflow$
through
\begin{align}\label{eq:weak_partial_b}
\spl \dens\partial_\bflow \dislag,\dislag' \spr :=
-\spl \dens\dislag,\partial_\bflow\dislag' \spr
-\spl\div(\dens\bflow)\dislag,\dislag'\spr
\end{align}
for $\dislag'\in \big(C^\infty_0(\domain)\big)^3$ and so $\partial_\bflow:=\dens^{-1}(\dens\partial_\bflow)$.
However, this assumption is not very restrictive as we can expect that the conservation of mass $\div(\dens\bflow)=0$ holds true.

\begin{lemma}\label{lem:HS}
If $\bflow$ satisfies the assumptions above, $\hs_{\bflow}$ is a well defined Hilbert space.
\end{lemma}
\begin{proof}
We only sketch the proof of completeness as the other Hilbert space properties are obvious. For $\bflow=0$ the statement is well-known. For general $\bflow$ let $(\dislag_n)$ be a Cauchy sequence in $\hs_{\bflow}$. 
Then $(\dislag_n)$ is also a Cauchy sequence in $\hs_0$, hence it converges to some $\overline{\dislag}\in \hs_0$ in $\hs_0$. 
We have to show that $\dislag\in\hs_b$ and $\lim_{n\to\infty}\|\partial_{\bflow}(\dislag_n-\dislag)\|_{L^2}=0$. 
Replacing $\dislag$ by $\dislag_n$ in \eqref{eq:weak_partial_b} and letting $n\to\infty$ shows that \eqref{eq:weak_partial_b} 
holds true with $\dislag$ replaced by $\overline{\dislag}$ for all $\dislag'\in \big(C^\infty_0(\domain)\big)^3$.
Replacing $\dislag$ in \eqref{eq:weak_partial_b} by $\dislag_n-\dislag_m$, letting $m\to\infty$, taking a supremum over all 
$\dislag'\in \big(C^\infty_0(\domain)\big)^3$ with $\|\dislag'\|_{\boldL^2}=1$ 
and using the Cauchy sequence property of $(\dislag_n)$, it is easy to see that 
$\lim_{n\to\infty}\|\dislag_n-\overline{\dislag}\|_{\boldL^2}=0$. 
\end{proof}
It is straightforward to see that the sesquilinear form $a(\cdot,\cdot)$ is well-defined and bounded on $(\hs\times \hsgrav) \times (\hs\times \hsgrav)$.
\begin{remark}
We note that the standard model S of \cite{ChristensenDEtal:96} for the sun assumes in the atmosphere $\dens(\boldx)=Ce^{-\alpha|\boldx|}$ with positive constants $C,\alpha$. However, in this article we consider only bounded domains $\domain$ and do not touch the topic of radiation conditions. Hence $\densl$ may be very small, but still positive - which poses no conflict with our assumptions.
\end{remark}

\subsection{Non-tangential flows}\label{subsec:nontangential}
Our analysis will solely deal with the sesquilinear form $\ses(\cdot,\cdot)$ defined in \eqref{eq:sesq-a}. Hence if we consider other configurations which lead to the same sesquilinear form, then they will also be covered by our analysis. In particular if the background flow $\bflow$ is non-tangential at the boundary ($\nv\cdot\bflow\neq0$), then the boundary integral
\begin{align*}
\int_{\partial\domain} (\nv\cdot\bflow) ((\omega+i\partial_\bflow+i\angvel\times)\dislag \cdot \ol{\dislag'}) \,\dd\boldx
\end{align*}
appears. However, if we impose the additional natural boundary condition
\begin{align*}
(\omega+i\partial_\bflow+i\angvel\times)\dislag &=0 \quad\text{at }\{\boldx\in\partial\domain\colon \nv(\boldx)\cdot\bflow(\boldx)\neq0\},
\end{align*}
then we end up with the sesquilinear form $\ses(\cdot,\cdot)$ in \eqref{eq:sesq-a} again.

\section{Analysis}\label{sec:analysis}
This section is devoted to the analysis of sesquilinear form~\eqref{eq:sesq-a}. In preparation we introduce some common functional framework.
For some notions it is more convenient to work with operators instead of sesquilinear forms. Thus for generic Hilbert spaces $(\hsgeneric, \spl\cdot,\cdot\spr_{\hsgeneric})$, $(\hsgeneric_1, \spl\cdot,\cdot\spr_{\hsgeneric_1})$, $(\hsgeneric_2, \spl\cdot,\cdot\spr_{\hsgeneric_2})$ we introduce the space $\blo(\hsgeneric_1,\hsgeneric_2)$ of bounded linear operators from $\hsgeneric_1$ to $\hsgeneric_2$ and set $\blo(\hsgeneric):=\blo(\hsgeneric,\hsgeneric)$.
For $\tilde\op\in\blo(\hsgeneric_1,\hsgeneric_2)$ we call $\tilde\op^*\in\blo(\hsgeneric_2,\hsgeneric_1)$ its adjoint, which is defined through $\spl \ugeneric,\tilde\op^*\ugeneric' \spr_{\hsgeneric_1} = \spl \tilde\op\ugeneric,\ugeneric' \spr_{\hsgeneric_2}$ for all $\ugeneric\in\hsgeneric_1, \ugeneric'\in\hsgeneric_2$.
We denote sesquilinear forms with lower case letters and operators with upper case letters. For a bounded sesquilinear form $\tilde \ses(\cdot,\cdot)$ let $\tilde \op\in \blo(\hsgeneric)$ be its Riesz representation, which is characterized by the relation
\begin{align}\label{eq:defRieszRep}
\spl \tilde \op \ugeneric, \ugeneric' \spr_\hsgeneric = \tilde \ses(\ugeneric,\ugeneric') \quad\text{for all}\quad \ugeneric,\ugeneric'\in\hsgeneric.
\end{align}
Vice-versa for $\tilde\op\in\blo(\hsgeneric)$ let $\tilde\ses(\cdot,\cdot)$ be the bounded sesquilinear form defined by the left-hand-side of \eqref{eq:defRieszRep}. The tildes in the previous definition were merely used to prevent a confusion with the sesquilinear form $\ses(\cdot,\cdot)$ defined in \eqref{eq:sesq-a}.
The variational equation \eqref{eq:variationaleq} can now be reformulated as operator equation
\begin{align*}
A(\dislag,\graveu)=(\tilde\boldf,0)
\end{align*}
(with $\tilde\boldf\in\hs$).

\begin{definition}
We say that $\tilde\op\in\blo(\hsgeneric)$ is coercive if $\inf_{\ugeneric\in\hsgeneric\setminus\{0\}} |\spl \tilde\op \ugeneric,\ugeneric \spr_\hsgeneric|/\|\ugeneric\|^2>0$. We say that $\tilde\op\in\blo(\hsgeneric)$ is weakly coercive, if there exists compact $K\in\blo(\hsgeneric)$ such that $\tilde\op+K$ is coercive.
We say that $\tilde\op\in\blo(\hsgeneric)$ is (weakly) $T$-coercive, if $T\in\blo(\hsgeneric)$ is bijective and $T^*\tilde\op$ is (weakly) coercive. The same coecivity properties are also attributed to the associated sesquilinear form $\tilde\ses$ defined by \eqref{eq:defRieszRep}.
\end{definition}
The following proposition follows easily from the Lax-Milgram lemma and Riesz theory:
\begin{proposition}
If $\tilde\op$ is weakly $T$-coercive, then $\tilde\op$ is a Fredholm operator with index zero. 
\end{proposition}

Our goal is to derive sufficient (and if possible also necessary) conditions on the parameters $\domain, \omega, \angvel, \sounds, \bflow, \dens, \pres, \grav, \damp$ to ensure that $\op$ is Fredholm. We will do so by proving weak $T$-coercivity of $\op$ with an explicitly defined operator $T$. Weak $T$-coercivity of $\op$ with explicit knowledge of $T$ is actually stronger than Fredholmness because it serves as a guideline for the construction of reliable discrete approximations, see, e.g., \cite{Halla:19StekloffAppr,Halla:19Tcomp}.

In Subsection \ref{subsec:helmholtz} we will introduce a suitable generalized Helmholtz decomposition. Subsequently we will define $T$ based on this decomposition.
The sesquilinear form \eqref{eq:sesq-a} has a very rich structure and admits several different phenomena. To present these in a clear manner we will introduce in Subsection~\ref{subsec:cowling} the Cowling approximation \eqref{eq:sesq-a-cow} of \eqref{eq:sesq-a}, which reduces the unknowns $(\dislag,\graveu)$ to $\dislag$. This is a reasonable step, because if $\dens\in W^{1,\infty}$, then the original sesquilinear form \eqref{eq:sesq-a} is Fredholm if and only if the Cowling Approximation \eqref{eq:sesq-a-cow} is so. For details see Subsection~\ref{subsec:cowling}. As a next step we will investigate in Subsections~\ref{subsec:bflow} and \ref{subsec:pres_grav} two special cases of parameters, which require different analysis techniques. We will discuss how to merge these two approaches for general parameters in Subsection~\ref{subsec:general}. Finally in Subsection~\ref{subsec:fullequations} we will state our results for the original sesquilinear form \eqref{eq:sesq-a}. 

At last, let us formulate the rather obvious injectivity of $\op$, which is caused by our modeling of the damping.
\begin{lemma}\label{lem:inj}
Let the assumptions of Subsection~\ref{subsec:assumptions} hold. Let $\hs$, $\hsgrav$ and $\ses(\cdot,\cdot)$ be as defined in \eqref{eq:hilspX}, \eqref{eq:hsgrav} and \eqref{eq:sesq-a}. If $\omega\neq0$, then the operator $\op$ induced by $\ses(\cdot,\cdot)$ is injective.
\end{lemma}
\begin{proof}
Let $(\dislag,\graveu)\in\ker A$. Then 
\begin{align*}
0=\Big|\Im \Big(\ses\big((\dislag,\graveu),(\dislag,\graveu)\big)\Big)\Big|
= |\omega| \spl \damp \dens \dislag,\dislag \spr
\geq |\omega|\dampl\densl \|\dislag\|^2_{\boldL^2}\end{align*}
and hence $\dislag=0$. We further compute 
\begin{align*}
0=\ses\big((\dislag,\graveu),(\dislag,\graveu)\big)
= \ses\big((0,\graveu),(0,\graveu)\big)
= \|\graveu\|^2_{\hsgrav}
\end{align*}
and conclude that $\graveu=0$. Thus $(\dislag,\graveu)=(0,0)$.
\end{proof}
Hence, if $\op$ is Fredholm and $\omega\neq0$, then it follows that $\op$ is bijective. 

\subsection{Generalized Helmholtz decomposition}\label{subsec:helmholtz}
In this section we will derive a generalized Helmholtz decomposition 
of the space $\hs$ adapted to our problem. 

Recall that a vector space $V$ is called the direct algebraic sum of 
subspaces $V_1,\dots,V_N\subset V$, denoted by 
\begin{align}\label{eq:alg_sum}
\hsgeneric=\bigoplus_{n=1,\dots,N}\hsgeneric_n
\end{align}
if each element $y\in \hsgeneric$ has a 
unique representation of the form $y=\sum_{n=1}^Ny_n$ with $y_n\in 
\hsgeneric_n$. We refer to \eqref{eq:alg_sum} as algebraic decomposition of $\hsgeneric$. 
Note that there exist associated projection operators $P_{\hsgeneric_n}:\hsgeneric\to \hsgeneric_n$, $y\mapsto y_n$ 
 with 
$\ran P_{\hsgeneric_n}=\hsgeneric_n$ and $\ker P_{\hsgeneric_n} = \bigoplus_{m=1,\dots,N,m\neq n}\hsgeneric_m$.  
\begin{definition}\label{def:topdecomp}
An algebraic decomposition \eqref{eq:alg_sum} of a Hilbert space 
$\hsgeneric$ is called a \emph{topological decomposition}, denoted by $\bigoplus^{\calT}$ if all associated projection 
operators $P_{\hsgeneric_n}$ are continuous. 
\end{definition}
Note that in a topological decomposition all subspaces 
$\hsgeneric_n = \bigcap_{m\neq n} \ker P_{\hsgeneric_m}$ are closed.
For  the following theorem let us introduce the short notation
\begin{align}\label{eq:qcoeff}
\qcoeff := \sounds^{-2}\dens^{-1}\grad\pres
\end{align}
and the embedding operator
\begin{align}
\opEmbed\dislag:=\dislag, \qquad \opEmbed\in\blo(\hs,\boldL^2).
\end{align}

\begin{theorem}\label{thm:decomposition}
Let $\bflow,\boldq\in L^{\infty}(\domain,\mathbb{R}^3)$,  
and let $\bflow$ satisfy the assumptions of Lemma \ref{lem:HS} such that $\hs$ is well-defined by \eqref{eq:hilspX}.
If $\bflow\neq0$ let $\domain$ be of class $C^{1,1}$ or convex.
Then $\hs$ admits a topological decomposition 
\begin{align}\label{eq:decomposition}
\hs = \hsV \oplus^\calT \hsW \oplus^\calT \hsZ
\end{align}
with the following properties: 
\begin{enumerate}
\item\label{it:Vcompactembedding}
$\hsV\subset\{\nabla\compvp: \compvp\in H^2(\domain)\mbox{ with }\frac{\partial \compvp}{\partial \nv}=0 \mbox{ on }\partial \domain\}$ 
is compactly embedded in $\boldL^2$, i.e.\
$\opEmbed\PV$ is compact.
\item\label{it:Wkernel}
$\hsW = \{\dislag \in \hs : \div \dislag + \qcoeff\cdot \dislag = 0\} $.
\item\label{it:Zfinite}
$\hsZ$ is finite-dimensional. 
\end{enumerate}
Moreover, if the domain $\domain$ is of class $C^{1,1}$ or convex, then there exists $\Creg\in (0,1)$ such that
\begin{align}\label{eq:regest-2}
\Creg^2 \|\nabla^\top\compv\|_{(L^2)^{3x3}}^2 -(1-\Creg^2) \|\compv\|_{L^2}^2 \leq \|\div\compv\|_{L^2}^2
\end{align}
for all $\compv\in\hsV$. 
If the domain $\domain$ is of class $C^{1,1}$ or convex and piecewise $C^{1,1}$, then for each $\eta\in W^{1,\infty}$ there exists a compact operator $K_\eta\in\blo(\hs)$ such that
\begin{align}\label{eq:multiplier}
\|\eta\div\compv\|_{L^2}^2=
\|\eta\nabla^\top\compv\|_{(L^2)^{3x3}}^2
+\spl K_\eta\compv,\compv \spr_{\hs}
\end{align}
for all $\compv\in\hsV$.
\end{theorem}

\begin{proof}
\emph{subspace $\hsW$:}
Let $\opR\in \blo(\hs,H^1)$ be defined by
\begin{align*}
\spl\opR\dislag,\compvp'\spr_{H^1} &:= \spl \dislag,\nabla\compvp' \spr
- \spl \qcoeff\cdot \dislag, \compvp' \spr,
\end{align*}
for all $\compvp'\in H^1, \dislag\in\hs$ and set 
$\hsW := \ker \opR$. Then by partial integration, property (\ref{it:Wkernel}) is satisfied. 

\emph{subspace $\hsZ$ and $P_{\hsZ}$:}
We further introduce $\opdecomp\in \blo(H^1)$ by 
\begin{align*}
\spl\opdecomp\compvp,\compvp'\spr_{H^1}
&:=\spl \nabla\compvp,\nabla\compvp' \spr
- \spl \qcoeff\cdot\nabla \compvp, \compvp' \spr
\end{align*}
for all $\compvp,\compvp'\in H^1$. Bounding the second term with the 
help of Young's inequality 
it is straightforward to see that $\opdecomp$ is weakly coercive and hence Fredholm. Thus $\ran\opdecomp$ is closed, 
and $\ker\opdecomp$ and $\ran\opdecomp^\bot$ are both finite-dimensional.
Note that formally $\opdecomp 
= \opR\circ \nabla$, and in particular 
$\ran \opdecomp \subset \ran \opR$. 
Let $\Pran$ be the $H^1$-orthogonal projection onto $\ran\opdecomp^\bot$. Thus $\ran\Pran\opR$ is finite dimensional, too. Let $\hsZ$ be a subspace of $\hs$ with
\begin{align}\label{eq:defhsZ}
\hsZ\subset\hs\colon \quad
\dim\hsZ=\dim\ran\Pran\opR \,\,\,\text{ and }\,\,\, \ran\Pran\opR|_{\hsZ}=\ran\Pran\opR.
\end{align}
It follows that $\Pran\opR|_{\hsZ}:\hsZ\to \ran\Pran\opR$ is a bijective linear mapping between finite dimensional spaces and hence 
boundedly invertible. Therefore, the equation 
\begin{align}\label{eq:eqforz}
\Pran\opR\compz=\Pran\opR\dislag
\end{align}
has a unique solution $\compz\in\hsZ$ for all $\dislag\in\hs$ depending continuously on $\dislag$, i.e.\ 
$\PZ:= (\Pran\opR|_{\hsZ})^{-1}\Pran\opR$ belongs to $\blo(\hs)$, and 
as $\dim \hsZ<\infty$ it is even compact. 
It follows directly from \eqref{eq:eqforz} that $\PZ\compz=\compz$ for $\compz\in\hsZ$, and hence $\PZ$ is a projection with $\ran\PZ=\hsZ$.

\emph{subspace $\hsV$ and $\PV$:}
For $\dislag\in\hs$ let $\compvp\in\ker B^\bot$ be the unique solution to
\begin{align}\label{eq:eqforvp}
\opdecomp\compvp=\opR(I_\hs-\PZ)\dislag
\end{align}
and set
\begin{align}\label{eq:defPV}
\PV\dislag:=\nabla\compvp, \qquad\hsV:=\ran\PV.
\end{align}
It follow that $\nabla\compvp\in\boldL^2$ and
\begin{subequations}
\begin{align*}
\Delta\compvp&=
\div((I_\hs-\PZ)\dislag)+\qcoeff\cdot ((I_\hs-\PZ)\dislag-\nabla\compvp)
\in L^2,\\
\nv\cdot\nabla\compvp&=0 \text{ at }\partial\domain.
\end{align*}
\end{subequations}
If $\bflow=0$, it already follows that $\nabla\compvp\in\hs$.
If $\bflow\neq0$, then we additionally demand $\domain$ to be either of class $C^{1,1}$ or convex (see \cite[Notation 2.1]{AmroucheBernardiDaugeGirault:98} for a definition of class $C^{1,1}$). This way standard regularity theory (see, e.g., \cite[Theorems~2.9, 2.17]{AmroucheBernardiDaugeGirault:98}) provides that $\compvp\in H^2$ and that there exists a constant $\Creg>0$ independent of $\compvp$ such that
\begin{align}\label{eq:regest-1}
\Creg^2 |\compvp|_{H^2}^2 -(1-\Creg^2)\|\nabla\compvp\|_{L^2}^2 \leq \|\Delta\compvp\|_{L^2}^2.
\end{align}
with the $H^2$-seminorm $|\compvp|_{H^2}:= (\sum_{j,k=1}^3\|\partial_{x_j}\partial_{x_k}v_0\|_{L^2}^2
)^{1/2}$.
So in this case $\partial_\bflow\nabla\compvp\in\boldL^2$ and $\nabla\compvp\in\hs$ follows as well. This shows that $\PV\in\blo(\hs)$.
Further, the embedding
$\opEmbed|_\hsV\colon \hsV\hookrightarrow\boldL^2$
is compact also for general Lipschitz domain $\domain$ (see, e.g., \cite{Weber:80} or \cite{Costabel:90}). Due to
\begin{align}\label{eq:relBR}
\opR \PV\dislag=\opR\nabla\compvp=\opdecomp\compvp
\end{align}
and \eqref{eq:eqforz} it follows that $\hsV\subset\ker\PZ$. Further, \eqref{eq:eqforvp}, $\hsV\subset\ker\PZ$ and \eqref{eq:relBR} yield that $\PV$ is indeed a projection. Since $\PZ$ is a projection, it also follows that $\hsZ\subset\ker\PV$ from \eqref{eq:eqforvp}.\\
To show \eqref{eq:multiplier} we apply \cite[Theorem~3.1.1.1]{Grisvard:85} or \cite[Lemma~2.11]{AmroucheBernardiDaugeGirault:98} to $\eta\compv\in H^1$ and use the compactness of Sobolev embeddings.

\emph{projection $\PW$:} 
At last we set
\begin{align}\label{eq:defPW}
\PW:=I_\hs-\PV-\PZ.
\end{align}
As we have already shown that $\PV$ and $\PZ$ are bounded projections 
with $\PV\PZ=\PZ\PV=0$, it follows that $\PW$ is a bounded projection 
with  $\hsV,\hsZ\subset\ker\PW$ and $\ran\PW\subset\ker\PV\cap\ker\PZ$. 
Therefore, it remains to show that $\ran\PW=\hsW$. First, suppose that 
$\dislag\in\ran\PW$, i.e.\ $\PW \dislag=\dislag$ and 
$\PZ\dislag = 0 =\PV\dislag$. Then $\nabla\compvp=0$ in \eqref{eq:defPV}, and hence $B\compvp=0$. Due to \eqref{eq:eqforvp} this implies 
$\opR\dislag=0$, i.e.\ $\dislag \in \hsW$.  Now suppose that 
$\dislag\in \hsW$, i.e.\ $\opR\dislag=0$. Then it follows from the 
definitions that $\PZ\dislag = 0$ and $\PV\dislag = 0$, and therefore 
$\dislag = \PW\dislag \in \ran \PW$.  This completes the proof 
that $\ran\PW=\hsW$.
\end{proof}

Note that in the case $\qcoeff=0$ we have $\hsZ=\{0\}$ since $\ran\opdecomp^\bot=\setC$ and $\ran\opR\bot\setC$.
Thus the decomposition \eqref{eq:decomposition} reduces 
to the well-known Helmholtz decomposition of $\dislag=\compv+\compw$ into a gradient function $\compv=\nabla\compvp$ and a divergence-free
function $\compw$, i.e.\ $\div\compw=0$. The third subspace $\hsZ$ is needed since for nonvanishing $\qcoeff$ the operators $\opdecomp$ 
and $\div+\qcoeff\cdot$ may not be surjective. Note that the role of $\hsZ$ is different from the role of the finite dimemsional space of harmonic fields (or differential forms) in a Hodge decomposition: e.g.\ we have $(\div+\qcoeff\cdot)\hsZ\neq \{0\}$. The decomposition 
$\dislag= \compv+\compw+\compz$ of a function $\dislag\in \hs_{\bflow}$ into its components in $\hsV$, $\hsW$ and $\hsZ$ coincides 
with the decomposition of $\dislag$ in $\hs_0$.
Of course the subspaces will change as $\hs_{\bflow} \subsetneq\hs_0$. 
A crucial property for the following analysis is that the spaces $\hsV$, $\hsW$, $\hsZ$ are indeed subspaces of $\hs_{\bflow}$ and that the projections onto the subspaces are continuous with respect to the norm of $\hs_{\bflow}$.

For $\dislag,\dislag'\in\hs$ we will often use the notation
\begin{align}\label{eq:notation_Helmholtz}
\begin{aligned}
&\compv:=\PV\dislag, &&\compw:=\PW\dislag, && \compz:=\PZ\dislag,\\
&\compv':=\PV\dislag',&&\compw':=\PW\dislag', &&\compz':=\PZ\dislag'
\end{aligned}
\end{align}
such that $\dislag = \compv+\compw+\compz$ and $\dislag'= \compv'+\compw'+\compz'$. 
At last we introduce the operator
\begin{align}\label{eq:defT}
T:=\PV-\PW+\PZ,
\end{align}
which switches the sign of $\compw$. 
We remark that this kind of ``sign-switch operator'' is commonly used in the analysis of the time-harmonic Maxwell equations, see, e.g., \cite{BuffaCostabelSchwab:02,Buffa:05,Halla:19StekloffAppr}.

\subsection{Cowling approximation}\label{subsec:cowling}
A common simplification, the so-called \emph{Cowling approximation}, of \eqref{eqs:stellar_osc} is to set $\graveu=0$ in \eqref{eq:stellar_osc_bd_dislag} and thus neglect the equations \eqref{eq:stellar_osc_grav}, \eqref{eq:stellar_osc_decay} for $\graveu$. The corresponding sesquilinear form is
\begin{align}\label{eq:sesq-a-cow}
\sescow(\dislag,\dislag'):=\ses\big((\dislag,0),(\dislag',0)\big).
\end{align}
There holds an injectivity result similar to Lemma~\ref{lem:inj}.
\begin{lemma}\label{lem:inj-cow}
Let the assumptions of Subsection~\ref{subsec:assumptions} hold. Let $\hs$ and $\sescow(\cdot,\cdot)$ be as defined in \eqref{eq:hilspX} and \eqref{eq:sesq-a-cow}. If $\omega\neq0$, then the operator 
$\opAcow$ induced by 
$\sescow(\cdot,\cdot)$ is injective.
\end{lemma}
\begin{proof}
Proceed as in the proof of Lemma~\ref{lem:inj}.
\end{proof}
We remark that if $\dens\in W^{1,\infty}$, then
\begin{align}\label{eq:prodrule}
\div(\dens\dislag)=\dens\div\dislag+\nabla\dens\cdot\dislag\in L^2
\end{align}
and the ``off-diagonal'' terms in $\ses\big((\dislag,\graveu),(\dislag',\graveu')\big)$ can be expressed as
\begin{subequations}\label{eq:offdiagonal}
\begin{align}
-\spl \grad\graveu,\dens\dislag'\spr&=\spl \graveu,\div(\dens\dislag')\spr=\spl \graveu,\dens\div\dislag'+\nabla\dens\cdot\dislag'\spr
=\spl K\graveu,\dislag'\spr_{\hs},\\
-\spl \dens\dislag,\grad\graveu'\spr&=\spl \div(\dens\dislag),\graveu'\spr=\spl \dens\div\dislag+\nabla\dens\cdot\dislag,\graveu'\spr
=\spl \dislag,K\graveu'\spr_{\hs}
\end{align}
\end{subequations}
with an operator $K\in\blo(\hsgrav,\hs)$. Due to the compactness
of the Sobolev embedding $\hsgrav\hookrightarrow L^2$ (recall that $L^2=L^2(\domain)$ and $\domain$ is bounded), $K$ is compact. Further, the equation for $\graveu$ itself is coercive ($\ses\big((0,\graveu),(0,\graveu)\big)=\|\graveu\|^2_{\hsgrav}$). Thus the original sesquilinear form $\ses$ in \eqref{eq:sesq-a} is Fredholm, if and only if the Cowling Approximation \eqref{eq:sesq-a-cow} is so. Hence to analyze \eqref{eq:sesq-a} it suffices to study \eqref{eq:sesq-a-cow}.

In the case $\dens\notin W^{1,\infty}$ the expansion \eqref{eq:prodrule} is not admissible and the ``off-diagonal'' terms in $\ses\big((\dislag,\graveu),(\dislag',\graveu')\big)$ cannot be rearranged as in \eqref{eq:offdiagonal}.
Thus no compactness property arises. However, the equation for $\graveu$ itself is still coercive. Hence we can build the Schur complement of $\ses(\cdot,\cdot)$ with respect to $\graveu$ and the corresponding sesquilinear form reads
\begin{align*}
\sesschur(\dislag,\dislag')=\sescow(\dislag,\dislag')-\sesadd(\dislag,\dislag')
\end{align*}
with a non-negative zeroth order term $\sesadd(\dislag,\dislag')$. As we will see in Subsection~\ref{subsec:fullequations} the analysis of $\sescow(\cdot,\cdot)$ needs only to be slightly adapted to treat $\sesschur(\cdot,\cdot)$. However, first we have to perform the analysis of $\sescow(\cdot,\cdot)$, which we will do in the following subsections.

\subsection{Background flow}\label{subsec:bflow}
In this subsection we consider constant pressure and gravitational potential ($\pres=\textrm{const}$ and $\grav=\textrm{const}$).
Under these additional assumptions we have
\begin{align}\label{eq:sesq-bflowdamp}
\begin{split}
\sescow(\dislag,\dislag')
&=\spl \sounds^2\dens \div \dislag, \div \dislag' \spr
-\spl \dens (\omega+i\partial_\bflow+i\angvel\times) \dislag, (\omega+i\partial_\bflow+i\angvel\times) \dislag' \spr\\
&-i\omega \spl \damp\dens \dislag, \dislag' \spr.
\end{split}
\end{align}

\begin{theorem}\label{thm:wTc-flow}
Let the assumptions of Subsection~\ref{subsec:assumptions} hold.
Let $\hs$, $\sescow(\cdot,\cdot)$ and $T$ be as defined in \eqref{eq:hilspX}, \eqref{eq:sesq-a-cow} and \eqref{eq:defT}.
Let $\domain$ be of class $C^{1,1}$ or convex and piece-wise $C^{1,1}$.
Let $\pres$, $\grav$ be constant, $\sounds, \dens \in W^{1,\infty}$ and $\omega\neq0$.
If
\begin{align*}
\|\sounds^{-1}\bflow\|_{\boldL^\infty}^2<1,
\end{align*}
then $\opAcow$ is weakly $T$-coercive.
\end{theorem}
\begin{proof}
Firstly we note that $T$ is self-inverse and hence bijective. Since $\pres$ is constant, we have $\qcoeff=0$ and $\hsZ=\{0\}$.
With the notation \eqref{eq:notation_Helmholtz}, note that 
$\langle T^*\opAcow \dislag,\dislag'\rangle_\hs = \sescow(\compv+\compw,\compv'-\compw')$.
We have to show that $T^*\opAcow=\op_1+\op_2$ can be split into the sum of a coercive operator $\op_1$ and a compact operator $\op_2$. 
To achieve this we insert into $\op_1$ an operator of the form 
\begin{align*}
\PV^*\left(\opEmbed^*\opEmbed+\frac{1}{4\delta}K^*K\right)\PV
\end{align*}
with compact operator $K\colon\hsV\to\hs$ and the scalar $\delta>0$ to  be chosen later, 
and insert the same operator with a minus sign into $\op_2$.  
More precisely, we define $\op_1, \op_2\in\blo(\hs)$ by
\begin{align*}
\spl \op_1\dislag,\dislag'\spr_\hs =\;& \spl \sounds^2\dens\div \compv,\div \compv'\spr
-\spl \dens i\partial_\bflow \compv,i\partial_\bflow \compv' \spr
+\spl \compv,\compv'\spr +\frac{1}{4\delta}\spl K\compv,K\compv' \spr_\hs\\
&+\spl \dens(\omega+i\partial_\bflow+i\angvel\times)\compw,(\omega+i\partial_\bflow+i\angvel\times)\compw' \spr
+i\omega \spl \damp\dens\compw,\compw'\spr\\
&-\spl \dens(\omega+i\partial_\bflow+i\angvel\times)\compw,i\partial_\bflow \compv' \spr
+\spl \dens i\partial_\bflow \compv,(\omega+i\partial_\bflow+i\angvel\times)\compw' \spr
\end{align*}
and
\begin{align*}
\spl \op_2\dislag,\dislag'\spr_\hs =\;&
-\spl \compv,\compv'\spr-\frac{1}{4\delta}\spl K\compv,K\compv' \spr_\hs\\
&-\spl \dens i\partial_\bflow \compv,(\omega+i\angvel\times)\compv' \spr
-\spl \dens (\omega+i\angvel\times)\compv,i\partial_\bflow \compv' \spr\\
&- \spl \dens (\omega+i\angvel\times)\compv,(\omega+i\angvel\times)\compv' \spr
-i\omega \spl \damp\dens\compv,\compv' \spr\\
&-\spl \dens(\omega+i\partial_\bflow+i\angvel\times)\compw,(\omega+i\angvel\times) \compv' \spr\\
&+\spl \dens(\omega+i\angvel\times) \compv,(\omega+i\partial_\bflow+i\angvel\times)\compw' \spr\\
&-i\omega \spl \damp\dens \compw,\compv' \spr
+i\omega \spl \damp\dens \compv,\compw' \spr
\end{align*}
for all $\dislag,\dislag'\in\hs$. Note that indeed $T^*\opAcow=\op_1+\op_2$ and that 
$\op_2$ is compact due to Theorem \ref{thm:decomposition}(\ref{it:Vcompactembedding}) and the compactness of $K$.  
To prove the coercivity of $\op_1$ we introduce a further parameter 
$\tau\in(0,\pi/2)$ and compute
\begin{align*}
\frac{1}{\cos\tau}
\Re\Big( e^{-i\tau\sign\omega} \spl \op_1 \dislag,\dislag \spr_\hs \Big) \hspace{-0.6mm}=\hspace{0.6mm}& \|\sounds\sqrt{\dens}\div \compv\|^2_{L^2}
-\| \sqrt{\dens}\partial_\bflow \compv \|^2_{\boldL^2}
+\|\compv\|_{\boldL^2}+\frac{1}{4\delta}\| K\compv\|_\hs^2\\
&+\| \sqrt{\dens}(\omega+i\partial_\bflow+i\angvel\times)\compw \|^2_{\boldL^2}
+|\omega|\tan\tau \| \sqrt{\damp\dens} \compw \|^2_{\boldL^2}\\
&-2\tan\tau\sign\omega\,\Im\big(\spl \dens i\partial_\bflow \compv,(\omega+i\partial_\bflow+i\angvel)\compw \spr_{\boldL^2}\big).
\end{align*}
Estimating the last term by the Cauchy-Schwarz inequality and the weighted Young inequality $2ab\leq (1-\epsilon)^{-1}a^2+(1-\epsilon)b^2$ 
with another parameter $\epsilon>0$, $a=\tan\tau \|\sqrt{\dens}\partial_\bflow \compv\|_{\boldL^2}$, and 
$b= \|\sqrt{\dens}(\omega+i\partial_\bflow+i\angvel\times)\compw\|_{\boldL^2}$ we obtain 
\begin{align*}
\frac{1}{\cos\tau}
\Re\Big( e^{-i\tau\sign\omega} \spl \op_1 \dislag,\dislag \spr_\hs \Big)
\hspace{-1.6mm}\geq\hspace{0.1mm}&\| \sounds\sqrt{\dens}\div \compv\|^2_{L^2}
-\left(1+(1-\epsilon)^{-1}\tan^2\tau\right)\| \sqrt{\dens}\partial_\bflow \compv \|^2_{\boldL^2}\\
&+\|\compv\|_{\boldL^2}+\frac{1}{4\delta}\|K\compv\|_\hs^2\\
&+\epsilon\| \sqrt{\dens}(\omega+i\partial_\bflow+i\angvel\times)\compw \|^2_{\boldL^2}
+|\omega|\tan\tau \| \sqrt{\damp\dens}\compw \|^2_{\boldL^2}.
\end{align*}
Now we choose the operator $K$ as $K:=K_\eta$ with $\eta:=\sounds\sqrt{\dens}$ and $K_\eta$ as in Theorem~\ref{thm:decomposition}. 
Due to this theorem the operator $K$ is indeed compact and we have 
\begin{align*}
\|\sounds\sqrt{\dens}\div\compv\|_{L^2}^2 = \|\sounds\sqrt{\dens}\nabla^\top\compv\|_{(L^2)^{3x3}}^2 +\spl K\compv,\compv\spr_\hs
\end{align*}
for each $\compw\in\hsV$.
We choose now $\epsilon$ and $\tau$ small enough such that
\begin{align*}
0<1-\left(1+(1-\epsilon)^{-1}\tan^2\tau\right)\|\sounds^{-1}\bflow\|_{\boldL^\infty}^2=:C_{\epsilon,\tau}.
\end{align*}
Bounding $\|\sqrt{\dens}\partial_\bflow\compv\|^2_{\boldL^2}$ by $\|\sounds^{-1}\bflow\|_{\boldL^\infty}^2\|\sounds\sqrt{\dens}\nabla^\top\compv\|_{(L^2)^{3\times3}}^2$ we can estimate
\begin{align*}
\|\sounds\sqrt{\dens}\div\compv\|^2_{L^2}&-(1+(1-\epsilon)^{-1}\tan^2\tau)\|\sqrt{\dens}\partial_\bflow\compv\|^2_{\boldL^2}\\
&\geq \soundsl^2\densl C_{\epsilon,\tau} |\compv|_{\boldH^1}^2 -|\spl K\compv,\compv\spr_\hs|\\
&\geq \soundsl^2\densl C_{\epsilon,\tau} |\compv|_{\boldH^1}^2 -\frac{1}{4\delta}\|K\compv\|_\hs^2 
-\delta\|\compv\|_\hs^2.
\end{align*}
We continue to estimate
\begin{align*}
\| \sounds\sqrt{\dens}\div \compv\|^2_{L^2}
&-\left(1+(1-\epsilon)^{-1}\tan^2\tau\right)\| \sqrt{\dens}\partial_\bflow \compv \|^2_{\boldL^2}
+\|\compv\|_{\boldL^2}+\frac{1}{4\delta}\|K\compv\|_\hs^2\\
&\geq \soundsl^2\densl C_{\epsilon,\tau} |\compv|_{\boldH^1}^2 +\|\compv\|_{\boldL^2}^2 -\delta\|\compv\|_\hs^2
\end{align*}
There exists a constant $C_V>0$ such that
\begin{align*}
\soundsl^2\densl C_{\epsilon,\tau} |\compv|_{\boldH^1}^2 +\|\compv\|_{\boldL^2}^2 \geq C_V\|\compv\|_\hs^2
\end{align*}
for each $\compv\in\hsV$. Thus
\begin{align*}
\| \sounds\sqrt{\dens}\div \compv\|^2_{L^2}
&-\left(1+(1-\epsilon)^{-1}\tan^2\tau\right)\| \sqrt{\dens}\partial_\bflow \compv \|^2_{\boldL^2}
+\|\compv\|_{\boldL^2}+\frac{1}{4\delta}\|K\compv\|_\hs^2\\
&\geq (C_V-\delta)\|\compv\|_\hs^2.
\end{align*}
Now we choose $\delta<C_V$.
The second part can be estimated using a weighted Young inequality and $\underline{\damp}>0$ to obtain
\begin{align*}
\epsilon\|\sqrt{\dens}(\omega+i\partial_\bflow+i\angvel\times)\compw\|^2_{\boldL^2}
&+|\omega|\tan\tau\|\sqrt{\damp\dens}\compw\|^2_{\boldL^2}\\
&\geq C_W(\|\partial_\bflow\compw\|^2_{\boldL^2} + \|\compw\|^2_{\boldL^2})
= C_W\|\compw\|_\hs^2
\end{align*}
for some $C_W>0$. 
Due to the equivalence of norms (see Definition~\ref{def:topdecomp}) the former estimates already yield the coercivity of $\op_1$.
\end{proof}
We remark that as long as $\dampl$ is positive, it can be arbitrarily small to satisfy the assumptions of the previous theorem. 

\begin{remark}[super-sonic flows]
The main assumption of Theorem~\ref{thm:wTc-flow} is
\begin{align*}
\|\sounds^{-1}\bflow\|_{\boldL^\infty}<1.
\end{align*}
It is a legimate question if this assumption can be further weakend. The answer is no --- at least for our kind of analysis. 
The situation can be examplified for the one-dimensional setting: Let $\domain=(-1,1)$ and $\angvel=0$. Then the sesquilinear form reads
\begin{align*}
\spl \dens\sounds^2(1-\sounds^{-2} b^2)\partial_x u,\partial_x u'\spr + \mathrm{low~order~terms}.
\end{align*}
Observe that if $|b|=\sounds$ on an open subset of $\domain$, then the principal part of the differential operator has 
an infinite-dimensional kernel, and therefore the differential operator cannot be Fredholm. Also if 
$|b|>\sounds$ on an open subset, the nature of the equation changes drastically. In particular, we loose uniform ellipticity if 
$b$ is continuous. 
Hence in this sense, the assumption is optimal.
\end{remark}

\subsection{Pressure and gravity}\label{subsec:pres_grav}
In this subsection we consider the opposite case to Subsection~\ref{subsec:bflow}. That is we consider the
case of no flow ($\bflow=0$), but non-homogeneous pressure $\pres$ and gravity $\grav$.
The core ingredient of the analysis in this case is to rewrite
\begin{align*}
\begin{aligned}
\spl \sounds^2\dens \div \dislag&, \div \dislag' \spr
+\spl \div \dislag, \nabla \pres \cdot \dislag' \spr
+\spl \nabla \pres \cdot \dislag, \div \dislag' \spr\\
=\;&\spl \sounds\sqrt{\dens} \div \dislag + \sounds^{-1}\dens^{-1/2} \nabla \pres \cdot \dislag,
\sounds\sqrt{\dens} \div \dislag' + \sounds^{-1}\dens^{-1/2} \nabla \pres \cdot \dislag' \spr\\
&-\spl \sounds^{-2}\dens^{-1}\nabla \pres \cdot \dislag, \nabla \pres \cdot \dislag' \spr,
\end{aligned}
\end{align*}
which is inspired by \cite[Chapter~6.5.4]{GoedbloedKeppensPoedts:19}.
Let
\begin{align}\label{eq:defmatltwo}
\begin{aligned}
\matgr(\boldx):=\;&
\dens(\omega I_{3\times3}+i\angvel\times)^H(\omega I_{3\times3}+i\angvel\times)\\
&+i\omega\damp\dens I_{3\times3}\\
&-\hess(\pres)+\dens\hess(\grav)
+\sounds^{-2}\dens^{-1} \nabla\pres\nabla\pres^\top.
\end{aligned}
\end{align}
This way we can express
\begin{align}\label{eq:ses-repr}
\sescow(\dislag,\dislag')=\spl \sounds^2\dens (\div\dislag+\sounds^{-1}\dens^{-1/2} \nabla \pres\cdot\dislag),\div\dislag'+\sounds^{-1}\dens^{-1/2} \nabla \pres\cdot\dislag' \spr
-\spl \matgr \dislag,\dislag' \spr.
\end{align}
The part $-\spl \matgr \dislag,\dislag' \spr$ contains only zero order terms, while the part $\spl \sounds^2\dens (\div\dislag+\sounds^{-1}\dens^{-1/2} \nabla \pres\cdot\dislag),\div\dislag'+\sounds^{-1}\dens^{-1/2} \nabla \pres\cdot\dislag' \spr$ is positive semi-definite. In this representation the sesqulinearform looks similarly to the sesquilinear form to an abstract time-harmonic wave equation with exterior derivative $\dd=\div+\sounds^{-1}\dens^{-1/2} \nabla \pres\cdot$. Further, from representation \eqref{eq:ses-repr} it follows that $\sescow(\cdot,\cdot)$ is already coercive. However, we will not exploit this observation. Instead we introduce the next theorem, which serves as preparation for the case of general parameters in Subsection~\ref{subsec:general}.
Recall that the numerical range of a matrix $M\in\setC^{3\times3}$ is defined by
\begin{align*}
\numran M := \{ \xi^H M \xi\colon \xi\in\setC^3, |\xi|_2=1\}.
\end{align*}

\begin{theorem}\label{thm:wTc-presgrav}
Let the assumptions of Subsection~\ref{subsec:assumptions} hold. Let $\hs$, $\sescow(\cdot,\cdot)$ and $T$ be as defined in \eqref{eq:hilspX}, \eqref{eq:sesq-a-cow} and \eqref{eq:defT}. Let $\bflow=0$ and $\omega\neq0$.
Then $\opAcow$ is weakly $T$-coercive.
\end{theorem}
\begin{proof}
We again use the notation \eqref{eq:notation_Helmholtz}. We decompose $T^*\opAcow=\op_1+\op_2$ whereby $\op_1, \op_2\in\blo(\hs)$ are defined by
\begin{align*}
\spl \op_1\dislag,\dislag'\spr_\hs =\;& \spl \sounds^2\dens\div \compv,\div \compv'\spr
+\spl \compv,\compv' \spr
+\spl \matgr \compw,\compw' \spr
+\spl \compz,\compz'\spr_\hs
\end{align*}
and
\begin{align*}
\spl \op_2\dislag,\dislag'\spr_\hs =\;&\spl \div\compv,\grad\pres\cdot\compv' \spr
+\spl \grad\pres\cdot\compv,\div\compv' \spr
+\spl \sounds^{-2}\dens^{-1}\; \nabla\pres\cdot\compv, \nabla\pres\cdot\compv' \spr\\
&-\spl \compv,\compv' \spr
-\spl \matgr\compv,\compv' \spr
+\spl \matgr\compv,\compw' \spr
-\spl \matgr\compw,\compv' \spr
\\& +\sescow(\compv+\compw,\compz')+\sescow(\compz,\compv'-\compw')
+\sescow(\compz',\compz')-\spl \compz,\compz'\spr_\hs
\end{align*}
for all $\dislag,\dislag'\in\hs$.
Note that indeed $T^*\opAcow=\op_1+\op_2$ and Operator $\op_2$ is compact due to Theorem~\ref{thm:decomposition}, parts (\ref{it:Vcompactembedding}) and (\ref{it:Zfinite}).
It remains to argue the coercivity of Operator $\op_1$.
To this end we note that due to $\dampl>0$ the set $\bigcup_{\boldx\in\domain}\numran\matgr(\boldx)$ is bounded away from zero and  contained in the closed salient sector spanned by $1$ and $e^{i\theta\sign\omega}$ for some $\theta\in(0,\pi)$. Hence
\begin{align*}
\frac{\sign\omega}{\sin\big((\pi-\theta)/2\big)} \Im ( &e^{i(\pi-\theta)/2\sign\omega} \spl \op_1\dislag,\dislag\spr_\hs)\\
&\geq  \soundsl^2\densl\|\div\compv\|_{L^2}^2 +\|\compv\|_{L^2}^2 +\densl\dampl|\omega|\|\compw\|_{\boldL^2}^2 +\|\compz\|_\hs^2\\
&\geq \min\{\soundsl^2\densl,1\}\min\{\densl\dampl|\omega|,1\}(\|\compv\|_\hs^2+\|\compw\|_\hs^2+\|\compz\|_\hs^2).
\end{align*}
The claim follows now due to the equivalence of norms (see Definition~\ref{def:topdecomp}).
\end{proof}

\subsection{General parameters}\label{subsec:general}
In this subsection we merge the analysis of Subsections~\ref{subsec:bflow} and \ref{subsec:pres_grav}.
To this end we introduce
\begin{align}
\label{eq:matkl}
\matkl&:=i\omega\dens\damp I_{3\times3} -\hess(\pres)+\dens\hess(\grav)+\sounds^{-2}\dens^{-1}\;\nabla\pres\nabla\pres^\top,\\
\label{eq:theta}
\theta&:=\max\Big\{0, \sup_{\boldx\in\domain}|\arg\numran \matkl|-\pi/2\Big\},
\end{align}
with
$\arg$ taking values in $(-\pi,\pi]$.

\begin{theorem}\label{thm:wTc-general}
Let the assumptions of Subsection~\ref{subsec:assumptions} hold. Let $\hs$, $\sescow(\cdot,\cdot)$, $T$ and $\theta$ be as defined in \eqref{eq:hilspX}, \eqref{eq:sesq-a-cow}, \eqref{eq:defT} and \eqref{eq:theta}.
Let $\domain$ be of class $C^{1,1}$ or convex and piece-wise $C^{1,1}$. Let $\sounds, \dens \in W^{1,\infty}$ and $\omega\neq0$.
If
\begin{align*}
\|\sounds^{-1}\bflow\|_{\boldL^\infty}^2<\frac{1}{1+\tan^2\theta},
\end{align*}
then $\opAcow$ is weakly $T$-coercive.
\end{theorem}

\begin{proof}
We proceed similarly as in the proof of Theorem~\ref{thm:wTc-flow}.
We note that $T$ is self-inverse and hence bijective.
We again use the notation \eqref{eq:notation_Helmholtz}.
Note that  $\langle T^*\opAcow \dislag,\dislag'\rangle_\hs = \sescow(\compv+\compw+\compz,\compv'-\compw'+\compz')$.
We have to show that $T^*\opAcow=\op_1+\op_2$ can be split into the sum of a coercive operator $\op_1$ and a compact operator $\op_2$. To achieve this we include an operator
\begin{align*}
\PV^*(\opEmbed^*\opEmbed+\frac{1}{4\delta}K^*K)\PV+\PZ^*\PZ
\end{align*}
into $\op_1$ and
\begin{align*}
-\PV^*(\opEmbed^*\opEmbed+\frac{1}{4\delta}K^*K)\PV-\PZ^*\PZ
\end{align*}
into $\op_2$ whereby the compact operator $K\colon\hsV\to\hs$ and the scalar $\delta>0$ will be chosen later
and define $\op_1, \op_2\in\blo(\hs)$ by
\begin{align*}
\spl \op_1\dislag,\dislag'\spr_\hs =\;& \spl \sounds^2\dens\div \compv,\div \compv'\spr
-\spl \dens i\partial_\bflow \compv,i\partial_\bflow \compv' \spr
+\spl \compv,\compv'\spr +\frac{1}{4\delta}\spl K\compv,K\compv' \spr_\hs\\
&+\spl \dens(\omega+i\partial_\bflow+i\angvel\times)\compw,(\omega+i\partial_\bflow+i\angvel\times)\compw' \spr
+\spl \matkl\compw,\compw'\spr\\
&-\spl \dens(\omega+i\partial_\bflow+i\angvel\times)\compw,i\partial_\bflow \compv' \spr
+\spl \dens i\partial_\bflow \compv,(\omega+i\partial_\bflow+i\angvel\times)\compw' \spr\\
&+\spl \compz,\compz' \spr_\hs
\end{align*}
and
\begin{align*}
\spl \op_2\dislag,\dislag'\spr_\hs =\;&
-\spl \compv,\compv'\spr-\frac{1}{4\delta}\spl K\compv,K\compv' \spr_\hs-\spl \compz,\compz' \spr_\hs\\
&+\spl \div\compv,\grad\pres\cdot\compv' \spr
+\spl \grad\pres\cdot\compv,\div\compv' \spr
+\spl \sounds^{-2}\dens^{-1}\; \nabla\pres\cdot\compv, \nabla\pres\cdot\compv' \spr\\
&-\spl \dens i\partial_\bflow \compv,(\omega+i\angvel\times)\compv' \spr
-\spl \dens (\omega+i\angvel\times)\compv,i\partial_\bflow \compv' \spr
-\spl \matgr \compv,\compv' \spr\\
&-\spl \dens(\omega+i\partial_\bflow+i\angvel\times)\compw,(\omega+i\angvel\times) \compv' \spr\\
&+\spl \dens(\omega+i\angvel\times) \compv,(\omega+i\partial_\bflow+i\angvel\times)\compw' \spr\\
&-\spl \matkl \compw,\compv' \spr
+\spl \matkl \compv,\compw' \spr\\
&+\sescow(\compv+\compw,\compz')+\sescow(\compz,\compv'-\compw')
+\sescow(\compz',\compz')
\end{align*}
for all $\dislag,\dislag'\in\hs$. Note that indeed $T^*\opAcow=\op_1+\op_2$ and that 
$\op_2$ is compact due to Theorem \ref{thm:decomposition}, parts (\ref{it:Vcompactembedding}) and (\ref{it:Zfinite}),
and the compactness of $K$.  
To prove the coercivity of $\op_1$ we introduce a further parameter $\tau\in(0,\pi/2)$ and compute
\begin{align*}
\frac{1}{\cos(\theta+\tau)}
\Re\Big( &e^{-i(\theta+\tau)\sign\omega} \spl \op_1 \dislag,\dislag \spr_\hs \Big) =\\
&\|\sounds\sqrt{\dens}\div \compv\|^2_{L^2}
-\| \sqrt{\dens}\partial_\bflow \compv \|^2_{\boldL^2}
+\|\compv\|_{\boldL^2}
+\frac{1}{4\delta}\| K\compv\|_\hs^2\\
&+\| \sqrt{\dens}(\omega+i\partial_\bflow+i\angvel\times)\compw \|^2_{\boldL^2}
+\frac{|\omega|\sin\tau}{\cos(\theta+\tau)} \| \sqrt{\damp\dens} \compw \|^2_{\boldL^2}\\
&+\|\compz\|_\hs^2
-2\tan(\theta+\tau)\sign\omega\,\Im\big(\spl \dens i\partial_\bflow \compv,(\omega+i\partial_\bflow+i\angvel)\compw \spr_{\boldL^2}\big).
\end{align*}
Estimating the last term by the Cauchy-Schwarz inequality and the weighted Young inequality $2ab\leq (1-\epsilon)^{-1}a^2+(1-\epsilon)b^2$ 
with another parameter $\epsilon>0$, $a=\tan(\theta+\tau) \|\sqrt{\dens}\partial_\bflow \compv\|_{\boldL^2}$, and 
$b= \|\sqrt{\dens}(\omega+i\partial_\bflow+i\angvel\times)\compw\|_{\boldL^2}$ we obtain 
\begin{align}\label{eq:abhierweiter}
\begin{split}
\frac{1}{\cos(\theta+\tau)}
\Re\Big( &e^{-i(\theta+\tau)\sign\omega} \spl \op_1 \dislag,\dislag \spr_\hs \Big) \geq\\
&\| \sounds\sqrt{\dens}\div \compv\|^2_{L^2}
-\big(1+(1-\epsilon)^{-1}\tan^2(\theta+\tau)\big)\| \sqrt{\dens}\partial_\bflow \compv \|^2_{\boldL^2}\\
&+\|\compv\|_{\boldL^2}+\frac{1}{4\delta}\|K\compv\|_\hs^2
+\|\compz\|_\hs^2\\
&+\epsilon\| \sqrt{\dens}(\omega+i\partial_\bflow+i\angvel\times)\compw \|^2_{\boldL^2}
+\frac{|\omega|\sin\tau}{\cos(\theta+\tau)} \| \sqrt{\damp\dens}\compw \|^2_{\boldL^2}.
\end{split}
\end{align}
Now we choose the operator $K$ as $K:=K_\eta$ with $\eta:=\sounds\sqrt{\dens}$ and $K_\eta$ as in Theorem~\ref{thm:decomposition}. 
Due to this theorem the operator $K$ is indeed compact and we have 
\begin{align*}
\|\sounds\sqrt{\dens}\div\compv\|_{L^2}^2 = \|\sounds\sqrt{\dens}\nabla^\top\compv\|_{(L^2)^{3x3}}^2 +\spl K\compv,\compv\spr_\hs
\end{align*}
for each $\compw\in\hsV$.
We choose now $\epsilon$ and $\tau$ small enough such that
\begin{align*}
0<1-\big(1+(1-\epsilon)^{-1}\tan^2(\theta+\tau)\big)\|\sounds^{-1}\bflow\|_{\boldL^\infty}^2=:C_{\epsilon,\tau,\theta}.
\end{align*}
Bounding $\|\sqrt{\dens}\partial_\bflow\compv\|^2_{\boldL^2}$ by $\|\sounds^{-1}\bflow\|_{\boldL^\infty}^2\|\sounds\sqrt{\dens}\nabla^\top\compv\|_{(L^2)^{3\times3}}^2$
we can estimate
\begin{align}\label{eq:estimateinproof}
\begin{split}
\|\sounds\sqrt{\dens}\div\compv\|^2_{L^2}&-\big(1+(1-\epsilon)^{-1}\tan^2(\theta+\tau)\big)\|\sqrt{\dens}\partial_\bflow\compv\|^2_{\boldL^2}\\
&\geq \soundsl^2\densl C_{\epsilon,\tau} |\compv|_{\boldH^1}^2 -|\spl K\compv,\compv\spr_\hs|\\
&\geq \soundsl^2\densl C_{\epsilon,\tau} |\compv|_{\boldH^1}^2 -\frac{1}{4\delta}\|K\compv\|_\hs^2 
-\delta\|\compv\|_\hs^2.
\end{split}
\end{align}
We continue to estimate
\begin{align*}
\| \sounds\sqrt{\dens}\div \compv\|^2_{L^2}
&-\big(1+(1-\epsilon)^{-1}\tan^2(\theta+\tau)\big)\| \sqrt{\dens}\partial_\bflow \compv \|^2_{\boldL^2}
+\|\compv\|_{\boldL^2}+\frac{1}{4\delta}\|K\compv\|_\hs^2\\
&\geq \soundsl^2\densl C_{\epsilon,\tau} |\compv|_{\boldH^1}^2 +\|\compv\|_{\boldL^2}^2 -\delta\|\compv\|_\hs^2
\end{align*}
There exists a constant $C_V>0$ such that
\begin{align*}
\soundsl^2\densl C_{\epsilon,\tau} |\compv|_{\boldH^1}^2 +\|\compv\|_{\boldL^2}^2 \geq C_V\|\compv\|_\hs^2
\end{align*}
for each $\compv\in\hsV$. Thus
\begin{align*}
\| \sounds\sqrt{\dens}\div \compv\|^2_{L^2}
&-\big(1+(1-\epsilon)^{-1}\tan^2(\theta+\tau)\big)\| \sqrt{\dens}\partial_\bflow \compv \|^2_{\boldL^2}
+\|\compv\|_{\boldL^2}+\frac{1}{4\delta}\|K\compv\|_\hs^2\\
&\geq (C_V-\delta)\|\compv\|_\hs^2.
\end{align*}
Now we choose $\delta<C_V$.
The second part can be estimated using a weighted Young inequality and $\underline{\damp}>0$ to obtain
\begin{align*}
\epsilon\|\sqrt{\dens}(\omega+i\partial_\bflow+i\angvel\times)\compw\|^2_{\boldL^2}
&+\frac{|\omega|\sin\tau}{\cos(\theta+\tau)}\|\sqrt{\damp\dens}\compw\|^2_{\boldL^2}\\
&\geq C_W(\|\partial_\bflow\compw\|^2_{\boldL^2} + \|\compw\|^2_{\boldL^2})
= C_W\|\compw\|_\hs^2
\end{align*}
for some $C_W>0$. 
Due to the equivalence of norms (see Definition~\ref{def:topdecomp}) the former estimates already yield the coercivity of $\op_1$.
\end{proof}
We observe that in contrast to Theorems~\ref{thm:wTc-flow} and \ref{thm:wTc-presgrav} the assumptions of Theorem~\ref{thm:wTc-general} depend on $\omega$ (apart from $\omega\neq0$). However it holds
$\lim_{|\omega|\to+\infty}\theta(\omega)=0$. Hence for large frequencies we asymptotically recover the very same assumption of Theorem~\ref{thm:wTc-flow}. Moreover, Theorem~\ref{thm:wTc-general} is indeed a generalization of Theorem~\ref{thm:wTc-flow} and Theorem~\ref{thm:wTc-presgrav}: If $\pres$ and $\grav$ are constant, then $\theta=0$. If $\bflow=0$, then the assumed estimate is a tautology.

We refer to Theorem~\ref{thm:wTc-general-strich} in the appendix for an adaptation of Theorem~\ref{thm:wTc-general} for rough data.

\subsection{Full equations}\label{subsec:fullequations}
Let us now discuss the original sesquilinear form $\ses(\cdot,\cdot)$ as defined in \eqref{eq:sesq-a}. To this end we introduce
\begin{align*}
\opQ &\in \blo(\boldL^2,\hsgrav), \quad \spl\opQ\boldxi,\graveu\spr_{\hsgrav}:=\spl\dens\boldxi,\grad\graveu\spr
\quad\text{for all}\quad \boldxi\in\boldL^2, \graveu\in \hsgrav,
\end{align*}
such that
\begin{align*}
\op &= \bpm \opAcow & \opEmbed^*Q^* \\ Q\opEmbed & (4\pi G)^{-1}I_{\hsgrav} \epm.
\end{align*}
Further for $\complsc\in\setC$, let
\begin{align}\label{eq:Tone}
T_1^\complsc: = \bpm \ol{\complsc} T & 0 \\ 0 & I_{\hsgrav} \epm.
\end{align}

\begin{theorem}\label{thm:wTc-full}
Let the assumptions of Subsection~\ref{subsec:assumptions} hold. Let $\hs$, $\hsgrav$, $\ses(\cdot,\cdot)$, $T^\complsc_1$ and $\theta$ be as defined in \eqref{eq:hilspX}, \eqref{eq:hsgrav}, \eqref{eq:sesq-a}, \eqref{eq:Tone} and \eqref{eq:theta}.
Let $\domain$ be of class $C^{1,1}$ or convex and piece-wise $C^{1,1}$. Let $\sounds, \dens \in W^{1,\infty}$ and $\omega\neq0$.
If
\begin{align*}
\|\sounds^{-1}\bflow\|_{\boldL^\infty}^2<\frac{1}{(1+\tan^2\theta)},
\end{align*}
then there exists $\complsc\in\setC$ such that $\op$ is weakly $T^\complsc_1$-coercive.
\end{theorem}
\begin{proof}
We proceed as in the proof of Theorem~\ref{thm:wTc-general}. Let $\tau$ be as therein and $\complsc:=e^{-i(\theta+\tau)\sign\omega}$. We decompose $T^*\opAcow-T^*\opEmbed^*\opQ^*\opQ\opEmbed=\tilde\op_1+\tilde\op_2$ with
\begin{align*}
\tilde\op_1:=\op_1+\PW^*\opEmbed^*\opQ^*\opQ\opEmbed\PW
\end{align*}
and
\begin{align*}
\tilde\op_2&:=\op_2
+(\PV+\PZ)^*\opEmbed^*\opQ^*\opQ\opEmbed(\PV+\PZ)\\
&-(\PV+\PZ)^*\opEmbed^*\opQ^*\opQ\opEmbed\PW
+\PW^*\opEmbed^*\opQ^*\opQ\opEmbed(\PV+\PZ)
\end{align*}
and $\op_1$ and $\op_2$ as in the proof of Theorem~\ref{thm:wTc-general}, such that
$(T^\complsc_1)^*\op=\boldA_1+\boldA_2$ with
$\boldA_1=\bpm \complsc\tilde\op_1 & 0 \\ 0 & (4\pi G)^{-1}I_{H^1_*} \epm$ and
$\boldA_2=\bpm \complsc\tilde\op_2 & \complsc T^*\opEmbed^*Q^* \\ Q\opEmbed & 0 \epm$.
Since $\PW^*\opEmbed^*\opQ^*\opQ\opEmbed\PW$ is positive semi definite, we can simply repeat the lines of the proof of Theorem~\ref{thm:wTc-general} to deduce that $\boldA_1$ is coercive.
The composed operator $\opQ\opEmbed$ and its adjoint $\opEmbed^*\opQ^*$ are compact as explained in Subsection~\ref{subsec:cowling}. The operator $\tilde\op_2$ is compact, because $\op_2$ is so and due to Theorem~\ref{thm:decomposition}, parts (\ref{it:Vcompactembedding}) and (\ref{it:Zfinite}). Hence $\boldA_2$ is compact and the claim follows.
\end{proof}
We refer to Theorem~\ref{thm:wTc-full-strich} in the appendix for an adaptation of Theorem~\ref{thm:wTc-full} for rough data.

\section{Conclusions}\label{sec:conclusion}
In this article we studied the time-harmonic linear equations of stellar oscillations without magnetic forces. 
We have proved the well-posedness of the equations under suitable assumptions which are essentially a smallness assumption on the $\boldL^\infty$-norm of the background flow (subsonic flows).  
By standard arguments this also implies finite sensitivity of the solution to the imperfect knowledge of coefficients or of the shape of the domain. 
Furthermore, our analysis provides a starting point for further research on the construction of reliable and provably convergent numerical methods.


\appendix

\section*{Appendix: Low regularity results}

In this appendix we show that the regularity assumptions of Theorems~\ref{thm:wTc-general} and \ref{thm:wTc-full} can be relaxed at the 
price of more restrictive bounds on $\|\bflow\|_{\boldL^{\infty}}$. 
If $\sounds$ and/or $\dens$ are not smooth, then the bounds on 
$\|\bflow\|_{\boldL^{\infty}}$ depend on $\soundsl$ 
and/or $\densl$. 

\begin{customthm}{3.10'} 
\label{thm:wTc-general-strich}
Let the assumptions of Subsection~\ref{subsec:assumptions} hold. Let $\hs$, $\sescow(\cdot,\cdot)$, $T$ and $\theta$ be as defined in \eqref{eq:hilspX}, \eqref{eq:sesq-a-cow}, \eqref{eq:defT} and \eqref{eq:theta}.
Let $\Creg$ be as in \eqref{eq:regest-2} and $\omega\neq0$.
Then $\opAcow$ is weakly $T$-coercive if one of the following lines holds true:\medskip\\
\noindent\resizebox{0.99\textwidth}{!}{
\begin{tabular}{l|l|l|l|l}  
$\#$ & domain $\domain$ & $\sounds$ & $\dens$ & estimate\\
\hline
a) & convex & & & $0<\Creg^2\soundsl^2\densl-(1+\tan^2\theta)\densu\|\bflow\|_{\boldL^\infty}^2$\\
b) & $C^{1,1}$ or convex $\&$ piece-wise $C^{1,1}$ & & & $0<\soundsl^2\densl-(1+\tan^2\theta)\densu\|\bflow\|_{\boldL^\infty}^2$\\
c) & $C^{1,1}$ or convex $\&$ piece-wise $C^{1,1}$ & $W^{1,\infty}$ & & $0<\densl-(1+\tan^2\theta)\densu\|\sounds^{-1}\bflow\|_{\boldL^\infty}^2$\\
d) & $C^{1,1}$ or convex $\&$ piece-wise $C^{1,1}$ & & $W^{1,\infty}$ & $0<\soundsl^2-(1+\tan^2\theta)\|\bflow\|_{\boldL^\infty}^2$\\
e) & $C^{1,1}$ or convex $\&$ piece-wise $C^{1,1}$ & $W^{1,\infty}$ & $W^{1,\infty}$ & $0<1-(1+\tan^2\theta)\|\sounds^{-1}\bflow\|_{\boldL^\infty}^2$
\end{tabular}
}
\end{customthm}
\begin{proof}
We proceed as in the proof of Theorem~\ref{thm:wTc-general} until Equation~\eqref{eq:abhierweiter} and continue the proof from thereon.
We choose $K$ as listed in the following table and choose $\epsilon>0$ and $\tau\in(0,\pi/2)$ such that $C_{\epsilon,\tau,\theta}>0$ for the given expressions $C_{\epsilon,\tau,\theta}$.\medskip\\
\noindent\resizebox{0.99\textwidth}{!}{
\begin{tabular}{l|l|l|l|l}
\# & $C=$ & $\eta$ & $K=$ & $C_{\epsilon,\tau,\theta}=$\\
\hline
a) & $\Creg^2$ & $1$ & $(\Creg^2-1)\PV^*\opEmbed^*\opEmbed\PV$
& $\Creg^2\soundsl^2\densl-(1+\big(1-\epsilon)^{-1}\tan^2(\theta+\tau)\big)\densu\|\bflow\|_{\boldL^\infty}^2$\\
b) & $1$ & $1$  & $K_{\eta=1}$ & $\soundsl^2\densl-(1+\big(1-\epsilon)^{-1}\tan^2(\theta+\tau)\big)\densu\|\bflow\|_{\boldL^\infty}^2$\\
c) & $1$ & $\sounds$  & $K_{\eta=\sounds}$ & $\densl-(1+\big(1-\epsilon)^{-1}\tan^2(\theta+\tau)\big)\densu\|\sounds^{-1}\bflow\|_{\boldL^\infty}^2$\\
d) & $1$ & $\dens^{1/2}$  & $K_{\eta=\dens^{1/2}}$ & $\soundsl^2-\big(1+(1-\epsilon)^{-1}\tan^2(\theta+\tau)\big)\|\bflow\|_{\boldL^\infty}^2$\\
e) & $1$ & $\sounds\dens^{1/2}$  & $K_{\eta=\sounds\dens^{1/2}}$ & $1-\big(1+(1-\epsilon)^{-1}\tan^2(\theta+\tau)\big)\|\sounds^{-1}\bflow\|_{\boldL^\infty}^2$
\end{tabular}
}\medskip\\
\noindent
Here
the operators $K_\eta$ are as in in Theorem~\ref{thm:decomposition}.
Due to this theorem the operator $K$ is compact in each case, 
and we have 
\begin{align*}
\|\eta\div\compv\|_{L^2}^2 \geq C\|\eta\nabla^\top\compv\|_{(L^2)^{3x3}}^2 +\spl K\compv,\compv\spr_\hs
\end{align*}
with the constants $C>0$ and the functions $\eta$ also listed in the table above.
Bounding $\|\sqrt{\dens}\partial_\bflow\compv\|^2_{\boldL^2}$ by $\overline{\dens}\|\bflow\|_{\boldL^\infty}^2\|\nabla^\top\compv\|_{L^2}^2$ 
in cases a) and b), by $\overline{\dens}\|\sounds^{-1}\bflow\|_{\boldL^\infty}^2\|\sounds\nabla^\top\compv\|_{L^2}^2$ in case c), 
by $\|\bflow\|_{\boldL^\infty}^2\|\sqrt{\dens}\nabla^\top\compv\|_{L^2}^2$ in case d), 
and by $\|\sounds^{-1}\bflow\|_{\boldL^\infty}^2\|\sqrt{\dens}\sounds\nabla^\top\compv\|_{L^2}^2$ in case e), 
we can estimate
\begin{align*}
\|\sounds\sqrt{\dens}\div\compv\|^2_{L^2}&-\big(1+(1-\epsilon)^{-1}\tan^2(\theta+\tau)\big)\|\sqrt{\dens}\partial_\bflow\compv\|^2_{\boldL^2}\\
&\geq \min\{1,\soundsl^2\}\min\{1,\densl\} C_{\epsilon,\tau} |\compv|_{\boldH^1}^2 -|\spl K\compv,\compv\spr_\hs|\\
&\geq \min\{1,\soundsl^2\}\min\{1,\densl\} C_{\epsilon,\tau} |\compv|_{\boldH^1}^2 -\frac{1}{4\delta}\|K\compv\|_\hs^2
-\delta\|\compv\|_\hs^2.
\end{align*}
We can now continue the proof of Theorem~\ref{thm:wTc-general} after \eqref{eq:estimateinproof}
with the previous estimate instead of \eqref{eq:estimateinproof}.
\end{proof}

For $\complsc\in\setC$, let
\begin{align}\label{eq:Ttwo}
T^\complsc_2: = \bpm I_\hs & \opEmbed^*Q^* \\ 0 & I_{H^1_*} \epm
T^\complsc_1
\bpm I_\hs & -\opEmbed^*Q^* \\ 0 & I_{H^1_*} \epm
\end{align}

\begin{customthm}{3.11'} 
\label{thm:wTc-full-strich}
Let the assumptions of Subsection~\ref{subsec:assumptions} hold. Let $\hs$, $\hsgrav$, $\ses(\cdot,\cdot)$, $T^\complsc_1$, $T^\complsc_2$ and $\theta$ be as defined in \eqref{eq:hilspX}, \eqref{eq:hsgrav}, \eqref{eq:sesq-a}, \eqref{eq:Tone}, \eqref{eq:Ttwo} and \eqref{eq:theta}. Let $\Creg$ be as in \eqref{eq:regest-2} and $\omega\neq0$.
If one of the following lines holds true:\medskip\\
\noindent\resizebox{0.99\textwidth}{!}{
\begin{tabular}{l|l|l|l|l}  
$\#$ & domain $\domain$ & $\sounds$ & $\dens$ & estimate\\
\hline
a) & convex & & & $0<\Creg^2\soundsl^2\densl-(1+\tan^2\theta)\densu\|\bflow\|_{\boldL^\infty}^2$\\
b) & $C^{1,1}$ or convex $\&$ piece-wise $C^{1,1}$ & & & $0<\soundsl^2\densl-(1+\tan^2\theta)\densu\|\bflow\|_{\boldL^\infty}^2$\\
c) & $C^{1,1}$ or convex $\&$ piece-wise $C^{1,1}$ & $W^{1,\infty}$ & & $0<\densl-(1+\tan^2\theta)\densu\|\sounds^{-1}\bflow\|_{\boldL^\infty}^2$\\
d) & $C^{1,1}$ or convex $\&$ piece-wise $C^{1,1}$ & & $W^{1,\infty}$ & $0<\soundsl^2-(1+\tan^2\theta)\|\bflow\|_{\boldL^\infty}^2$\\
e) & $C^{1,1}$ or convex $\&$ piece-wise $C^{1,1}$ & $W^{1,\infty}$ & $W^{1,\infty}$ & $0<1-(1+\tan^2\theta)\|\sounds^{-1}\bflow\|_{\boldL^\infty}^2$
\end{tabular}
}\medskip\\
then there exists $\complsc\in\setC$ such that $\op$ is weakly $T^\complsc_2$-coercive.
If in addition $\dens\in W^{1,\infty}$, then $\op$ is also weakly $T^\complsc_1$-coercive.
\end{customthm}
\begin{proof}
We exploit the Schur factorization
\begin{align*}
\op&= \bpm I_\hs & \opEmbed^*Q^* \\ 0 & I_{\hsgrav} \epm
\bpm \opAcow - \opEmbed^*Q^*Q\opEmbed & 0 \\ 0 & (4\pi G)^{-1}I_{\hsgrav} \epm
\bpm I_\hs & 0 \\ Q\opEmbed & I_{\hsgrav} \epm
\end{align*}
such that
\begin{align*}
(T^\complsc_2)^*A &=
\bpm I_\hs & \opEmbed^*Q^* \\ 0 & I_{\hsgrav} \epm
\bpm \complsc T^*(\opAcow - \opEmbed^*Q^*Q\opEmbed) & 0 \\ 0 & (4\pi G)^{-1}I_{\hsgrav} \epm
\bpm I_\hs & 0 \\ Q\opEmbed & I_{\hsgrav} \epm \\
&=
\bpm I_\hs & \opEmbed^*Q^* \\ 0 & I_{\hsgrav} \epm
\bpm \sigma \tilde\op_1 & 0 \\ 0 & (4\pi G)^{-1}I_{\hsgrav} \epm
\bpm I_\hs & 0 \\ Q\opEmbed & I_{\hsgrav} \epm \\
&\phantom{=}+
\bpm I_\hs & \opEmbed^*Q^* \\ 0 & I_{\hsgrav} \epm
\bpm \sigma \tilde\op_2 & 0 \\ 0 & 0 \epm
\bpm I_\hs & 0 \\ Q\opEmbed & I_{\hsgrav} \epm
\end{align*}
with $\tilde\op_1$ and $\tilde\op_2$ as in the proof of Theorem~\ref{thm:wTc-full}.
We follow the lines of the proofs of Theorems~\ref{thm:wTc-general-strich} and \ref{thm:wTc-full} to deduce that $\tilde\op_1$ is coerive and $\tilde\op_2$ is compact. Hence the first summand of the former decomposition is coercive and the second summand is compact. In the case $\dens\in W^{1,\infty}$ the operator $\opQ\opEmbed$ is compact and so $T^\complsc_1$ and $T^\complsc_2$ differ only by a compact operator. Hence the claim is proven.
\end{proof}

\bibliographystyle{amsplain}
\bibliography{short_biblio}
\end{document}